\documentclass{mcom-l}

\usepackage{amssymb, pseudocode}

\usepackage[cmtip,all]{xy}

\newtheorem{theorem}{Theorem}[section]
\newtheorem{prop}[theorem]{Proposition}
\newtheorem{lemma}[theorem]{Lemma}

\theoremstyle{definition}
\newtheorem{definition}[theorem]{Definition}
\newtheorem{example}[theorem]{Example}

\theoremstyle{remark}

\numberwithin{equation}{section}

\newcommand{\Fq}{\mathbb{F}_q}
\newcommand{\F}{\mathcal{F}}

\begin{document}

\title{Computations in Cubic Function Fields of Characteristic Three}


\author{Mark Bauer}
\address{University of Calgary, Calgary AB T2L 1N4, CANADA}
\email{mbauer@math.ucalgary.ca}

\author{Jonathan Webster}
\address{Bates College,  Lewiston ME 04240, USA}
\email{jwebster@bates.edu}


\date{}

\dedicatory{}

\begin{abstract}

This paper contains an account of arbitrary cubic function fields of characteristic three.  We define a standard form for an arbitrary cubic curve and consider its function field.  By considering an integral basis for the maximal order of these function fields, we are able to calculate the field discriminant and the genus.  We also describe the splitting behavior of any place, and give composition and reduction algorithms for arithmetic in the ideal class group.\\    

\end{abstract}

\maketitle


\section{Introduction}

Calculating invariants of a field and its maximal order remains one of the central problems in computational number theory.  Motivated by hyperelliptic curve cryptography and well-studied cubic number fields, a host of authors have researched computational properties of cubic function fields.  From calculating fundamental units \cite{fun_unit}, to computing in the ideal class group \cite{bauer}, to tabulating \cite{rozen}, to describing and classifying  arbitrary cubic function fields \cite{exp_cff,aacff}, the results (mostly) exclude characteristic three. 

In the case in which characteristic three is considered, it is often through generic methods.  The function field analogue of the Round 2 algorithm, algebraic methods involving desingularization, or using Groebner basis to do ideal arithmetic, all may be applied to the problems considered in this paper.  However, these methods are often impractical and can make it difficult to understand how the basic invariants of a function field arise from the defining curve.   For elliptic curves and hyperelliptic curves, we may compute the desired quantities directly from the defining curve and the underlying finite field.  For cubic function fields in characteristic greater than three, much progress has been made in this regard; our goal is to extend these computations to characteristic three.  It is important to mention that work to this end has also been undertaken in \cite{tobias}.  However, the aim of our project is slightly different --- as opposed to developing a coherent theory for signatures across different characteristics, we have chosen to completely analyze all cubic function fields in characteristic three and develop the associated algorithms for computations.

We begin by developing the basic invariants of cubic function fields.  Section 2 defines function fields and states the standard model that will be used to define the field.  Section 3 contains the calculation of the integral basis and the field discriminant for the fixed model.  The following section describes the splitting behavior for places and this is used in Section 5 to calculate the genus.  This will conclude the calculation of basic invariants.   We review the relationship between the ideal class group and the Jacobian to motivate an explicit means of doing computations in the ideal class group.  We state integral basis for the prime ideals and their powers in Section 7.   Using this basis motivates arbitrary ideal arithmetic which is given over the next two sections.  Finally, we state an algorithm to do composition and reduction in the ideal class group and conclude with an example computation.


\section{Standard Form}

As there are many good introductions to algebraic function fields (for example \cite{sticht,dino}), we will only seek to clarify the notation used in this paper.  As usual, let $\mathbb{F}_q$ be a finite field and $\mathbb{F}_q[x]$ and $\mathbb{F}_q(x)$ be the ring of polynomials and the field of rational functions, respectively, in $x$ over $\mathbb{F}_q$.  An algebraic function field is a finite extension $\mathcal{F}$ of $\mathbb{F}_q(x)$; it thus may be written as $\mathcal{F} = \mathbb{F}_q(x,y)$ with  $y$ a root of $H(T)$, where $H(T)$ is an irreducible monic polynomial in $({\mathbb{F}_q}[x])[T]$ of degree $n = [\mathcal{F}:\mathbb{F}_q(x)]$.   Hence to study cubic function fields, we consider affine planar curves $H(x,T)$, where $H \in \mathbb{F}_q[x,T]$ is a bivariate polynomial  that is absolutely irreducible and of degree three in $T$.  

When $\textrm{char}(\mathbb{F}_q) \neq 3$,  cubic function fields may be studied by examining the standard form for the defining polynomial which is given by $T^3 - AT + B = 0$ with $A, B \in \mathbb{F}_q[x]$ provided there is no non-constant $Q \in \mathbb{F}_q[x]$ such that  $Q^2|A$ and $Q^3|B$ (see \cite{aacff} for details).  By considering these birationally equivalent curves, it is possible to study arbitrary curves as a two-parameter family.  Our goal will be to find a model which gives a similar two-parameter family in characteristic three.  Henceforth, let $\textrm{char}(\mathbb{F}_q) = 3$ unless explicitly stated otherwise.  

Write $H(x,T) = ST^3 + UT^2 + VT + W$ with $S,U,V,W \in \mathbb{F}_q[x]$ and $SW \neq 0$.  If $U=V=0$, then the function field associated with this curve is purely inseparable, and hence isomorphic to the rational function field (see Proposition III.9.2 of \cite{sticht}).  We thus require $U \neq 0$ or $V \neq 0$ to avoid this degenerate case.  

If $U = 0$ then making the polynomial monic yields a curve in of the form $T^3 - AT + B = 0$.  In terms of the original parameters,  $A = -SV$ and $B = S^2W$.  Otherwise, $U \neq 0$ and transform via $T \rightarrow (T + V)/U$ to eliminate the linear term.  Considering the monic, integral, reciprocal polynomial, we get a curve in the form $T^3 - AT + B = 0$.  If $N = S - U^2V+U^3W$, then  $A = -N^2U^2$ and $B = N^2S$ in terms of the parameters of the original curve.

Henceforth, we will restrict our attention to curves of the form $T^3 - AT + B = 0$.  In what follows we use the fact that $T \rightarrow T + i$ yields the birationally equivalent curve $T^3 -AT + (i^3 - iA + B) = 0$.  Our goal will be two-fold --- to minimize both the repeated factors dividing $A$ and the degree of $B$.  

If there is a polynomial $Q \in \mathbb{F}_q[x]$ such that $Q^2|A$ and $Q^3|(i^3 -iA + B)$ for some $i \in \mathbb{F}_q[x]$ then it is possible to consider the curve given by
\begin{equation}\label{sing_remove}T^3 - \left(\frac{A}{Q^2} \right)T +\left(\frac{i^3 - iA + B}{Q^3}\right) = 0.\end{equation}
If the characteristic is not three, the existence of a $Q$ such that $Q^2|A$ and $Q^3|B$ implies that $y/Q$ is integral.  Likewise, in characteristic three the existence of $Q$ and $i$ implies $(y+i)/Q$ is integral and has a minimal polynomial (\ref{sing_remove}).  Hence this is the appropriate generalization from the case where the characteristic is different from three.  Also, in each case $Q$ corresponds  to removable singularities that preserve the shape of the model for the given curve.  

To find $Q$ and $i$, it is sufficient to check irreducible polynomials $P$ such that $P^2|A$.  Begin by writing $i = i_0 + i_1P + i_2P^2$ with $i_0, i_1, i_2 \in \Fq[x]$ with degree less than that of $P$.  Since only $i_0$ affects the congruence $i^3 -iA + B \equiv 0 \pmod{P^3}$, we solve $i_0^3 + B \equiv 0 \pmod{P}$.  It then becomes a matter of checking whether $i_0^3 - i_0A + B \equiv 0 \pmod{P^3}$.  If the congruence holds, redefine $A$ as $A/P^2$ and $B$ as $(i^3 - iA + B)/P^3$.  This process may be repeated as needed.  

  We now turn our focus to reducing the degree of $B$.  If $3 | \deg B$ and $ 2\deg B > 3\deg A $, we can write  $B(x) = b_{3n}x^{3n} + b_{3n-1}x^{3n-1} +  \cdots + b_0$ with $b_{3n} \neq 0$.  Then consider the linear transformation  $T \rightarrow T - (b_{3n})^{1/3}x^n$ (note that  $b_{3n}^{1/3} \in \mathbb{F}_q$ because $\mathbb{F}_q$ is perfect). Under this map we get a new curve 
\begin{equation*} T^3 - A(x)T + b_{3n-1}x^{3n-1} + \ldots + b_0 + A(x)b_{3n}^{1/3}x^n = 0, \end{equation*}
where the polynomial $ b_{3n-1}x^{3n-1} +  \ldots + b_0 + A(x)(b_{3n})^{1/3}x^n$ has a lower degree than $B(x)$.  By repeating this procedure it is possible to force the curve to satisfy one of the following two distinct criteria:
\begin{equation}\label{wildram} 3 \nmid \deg B \mbox{ \quad and \quad } 2\deg B > 3\deg A  \end{equation}
or
\begin{equation}\label{tame} 2\deg B \leq 3\deg A. \end{equation}

\begin{definition} A curve is said to be a \emph{standard model} (or in \emph{standard form}) for a cubic function field if it is of the form $T^3 - AT + B = 0$ with no $Q, i\in \mathbb{F}_q[x]$ such that $Q^2|A$ and $Q^3 | i^3 - iA + B$ and satisfies either (\ref{wildram}) or (\ref{tame}). 
\end{definition}

It will also be useful to have a simple criterion to detect singularities.  

\begin{prop}
The curve  $T^3 - A(x)T + B(x) = 0$ is nonsingular if and only if $\deg d=0$ where $d = {\rm gcd}(A(x), A'(x)^3B (x)+ B'(x)^3)$. 
\end{prop}

\begin{proof}
A singular point $(a,b) \in \overline{\mathbb{F}}_q^2$ satisfies the following three equations.  
\begin{eqnarray}
 \label{sing1} b^3 - A(a)b + B(a) = 0. \\
 \label{sing2}               A(a) = 0. \\
 \label{sing3}   - A'(a)b + B'(a) = 0.
\end{eqnarray}

From (\ref{sing2}), $a$ is a root of $A(x)$, and combined with (\ref{sing1}) we see that $-b^3 = B(a)$.  Cubing (\ref{sing3}) we  get $-A'(a)^3b^3 + B'(a)^3 = 0$ which implies $A'(a)^3B(a) + B'(a)^3 = 0$.  Thus $a$ is a common root of $A(x)$ and $A'(x)^3B(x) + B'(x)^3$. 

For the converse let $a$ be a common root of $A(x)$ and $A'(x)^3B (x)+ B'(x)^3$.  Since $a$ is a root of $A(x)$,  (\ref{sing2}) is satisfied.  Since $\overline{\mathbb{F}_q}$ is perfect, we can find $b$ such that $b^3 = -B(a)$ in order to satisfy (\ref{sing1}).  With (\ref{sing1}) and (\ref{sing2}) satisfied, it is clear that (\ref{sing3}) is also satisfied by the above construction.
\end{proof}

Note that for large $q$ we do not expect a curve selected in standard form to be singular.  That is, if singularity is detected by $\deg d$ not being $0$, then it is a question of when two ``random" polynomials are relatively prime.  This happens with probability roughly $1 - 1/q$.   

Calculating the standard form and the integral basis, as well as finding the field discriminant and the genus are all closely related to singularity.  The square factors removed from $A$ in the conversion to standard form correspond to singular points, which simplifies future calculations. In fact, if the standard form is nonsingular then $\{1, y, y^2\}$ is an integral basis for the maximal order (see Prop 5.10 of \cite{dino}).  We now know that $D = {\rm disc}(y) = A^3$ (for the reader who is unfamiliar with this concept, it will be defined more formally below).  In the next section we will show that the square-free factorization of $d = {\rm gcd}(A(x), A'(x)^3B (x)+ B'(x)^3)$ is $I = {\rm ind}(y)$.   With $D$ and $I$ in hand, we will have $\Delta = {\rm disc}(\mathcal{F})$.

Knowing that $D = A^3$ and that $\Delta$ differs from $D$ by square factors is enough to determine when $\mathcal{F}$ is an Artin-Schreier extension.  

\begin{theorem}
$\mathcal{F}$ is an Artin-Schreier extension if and only if $A(x)$ is a square.
\end{theorem}
\begin{proof}
Cubic extensions are Galois (which is to say an Artin-Schreier extension in characteristic 3) if and only if their discriminant is a square.  In order to have a square discriminant, $A(x)$ must be a square.  Conversely, if $T^3 - T = f/g$ with $f,g \in \mathbb{F}_q[x]$ is an Artin-Schreier extension then $T^3 - g^2T = fg^2$ is an integral model for this equation.  By renaming, we have $A(x) = g(x)^2$ a square.  
\end{proof}


\section{Integral basis and field discriminant}

We will follow Chapter 2 Section 17 of \cite{dandf} to find an integral basis for $\mathcal{O}_\mathcal{F}$, the integral closure of $\mathbb{F}_q[x]$ in $\mathcal{F}$.  Recall that the powers $1,\ y,\ y^2$ form a basis of the $\mathbb{F}_q(x)$-vector space $\mathcal{F}$.  An $\mathbb{F}_q(x)$-basis given by $\{\alpha_0, \alpha_1, \alpha_2 \}$ is triangular if $\alpha_0$  and $\alpha_1$ are an $\mathbb{F}_q(x)$-linear combination of $1$ and $1, y$, respectively.  The three conjugate mappings taking $y$ to the three  roots $ y = y^{(0)}, y^{(1)}, y^{(2)}$ defines for every $\alpha \in \mathcal{F}$ its three conjugates $\alpha = \alpha^{(0)}, \alpha^{(1)},\alpha^{(2)}$, and allows for the following definition of the discriminant of three elements:
\begin{equation*} \rm{disc}(\alpha_0, \alpha_1, \alpha_{2}) = \rm{det}(\alpha_i^{(j)})^2_{0\leq i, j, \leq 2} \in \mathbb{F}_q(x). 
\end{equation*}

The ring $\mathcal{O}_\mathcal{F}$ always admits a triangular basis, one element of which is (obviously) in $\mathbb{F}_q^*$.  The discriminant of $\mathcal{F}/\mathbb{F}_q(x)$ is $\rm{disc}(\mathcal{F}) = \rm{disc}(\alpha_0, \alpha_1, \alpha_{2})$ where $\{\alpha_0, \alpha_1,\alpha_{2} \}$ is an integral basis of  $\mathcal{F}/\mathbb{F}_q(x)$, i.e. a basis for $\mathcal{O}_\mathcal{F}$.  For any element $\alpha \in \mathcal{F}$, the index of $\alpha$ satisfies  $\rm{disc}(\alpha) = \rm{ind}(\alpha)^2\rm{disc}(\mathcal{F})$, which will be crucial in determining a basis for $\mathcal{O}_\mathcal{F}$.

 Writing down a basis in triangular form, we will be able to deduce restrictions on the elements of the basis simply by using the fact that they are integral.  These restrictions arise naturally by examining the minimal polynomial of each element.  Following \cite{dandf}, we choose the product of the latter two basis elements to be in $\mathbb{F}_q[x]$.  Consider the integral  basis given by 
\begin{equation*} \left[1, \frac{y - i}{I_1}, \frac{ c + by + y^2}{I_2} \right] = [1, \rho, \omega ] \end{equation*}
with $I_1,I_2, i, c, b \in \mathbb{F}_q[x]$  (the choice to reuse $i$ will become clear).  As mentioned before, the integral basis construction was a motivation for the choice of the standard model; in particular, the minimal polynomial of $\rho$ is given by
\begin{equation*} \rho^3 -\frac{A}{I_1^2}\rho + \frac{i^3 - iA + B}{I_1^3} = 0 .\end{equation*}
Since this is an integral equation in $\rho$, it must be that $I_1^2 |A$ and $I_1^3 |i^3 -iA +B$.   This is the same criterion as (\ref{sing_remove}).  Thus the reduction to standard form forces $I_1 = 1$.  Now consider $\rho\omega \in \mathbb{F}_q[x]$ to get additional criteria on $i,b,c, $ and $I_2$:   
\begin{equation*} \rho\omega = \frac{(b-i)y^2 + (A-ib+c)y - (ic + B)}{I_2}. \end{equation*}
This implies $i=b$, $c = i^2 - A$, and $I_2|ic + B$.  Combining the last two statements, $I_2|i^3 - iA + B$.  Rewrite $\omega$ as $(y^2 + iy + i^2 - A)/I_2$ and consider its minimal polynomial to get our final criterion:  
\begin{equation*} \omega^3 + \frac{A}{I_2}\omega^2 - \frac{(i^3 -iA + B)^2}{I_2^3} = 0. \end{equation*}

This gives $I_2 | A$ and $I_2^3|(i^3 -iA + B)^2$.  Choosing $i$ such that $I_2$ is of maximal degree yields the basis. This observation will in fact force $I_2$, which is the index of $y$,  to be square-free.  From this point forward the subscript of $I_2$ will be dropped and the index of $y$ will be denoted $I$.   

\begin{prop}\label{t:squarefree}  A curve in standard from has $I = {\rm ind}(y)$ being square-free. \end{prop}
\begin{proof}
Let $P \in \Fq(x)$ be irreducible such that  $P|I$, the index.  If $v_P(A) = 1$ then $v_P(I) = 1$. So assume $v_P(A) \geq 2$ and consider $i$ such that $I^3 |( i^3 -iA + B)^2$.  If $v_P( i^3 -iA + B ) = 2$ then $v_P(I) = 1$.  However, if $v_P(A) \geq 2$ and $v_P( i^3 -iA + B ) \geq 3$ then the curve is not in standard form.
\end{proof}

 Having established that the index is square free, it is possible to calculate $i$.  Since $i$ is unique modulo $I$, $i$ is determined by its residue class modulo each distinct prime dividing $I$.   For each irreducible polynomial $P|I$, we solve $(i^3 -iA + B)^2 \equiv 0 \pmod{P^3}$ and constuct the solution using the Chinese Remainder Theorem.  As we did when removing singularities, we write $ i = i_0 + i_1P + i_2P^2$ and solve congruence equations modulo $P$, $P^2$, and $P^3$.  

With the index of  $y$ calculated, it is straightforward to determine the discriminant of the function field simply by noting that $D=A^3$ and hence $\Delta=A^3/I^2$.   It is well known that the field discriminant is closely tied to ramification.  This is developed more fully in the next section.  

Letting $A = EI$ and $FI^2 = i^3 - iA + B$, we have the following identities for various products of integral basis elements:
\begin{equation*} \rho^2 = I\omega + A, \quad \omega^2=  - E\omega - F\rho, \quad \rho\omega = -FI .\end{equation*}


\section{Splitting of Places}
  
The places of $\mathbb{F}_q(x)$ consist of finite places, identified with the monic irreducible polynomials in $\mathbb{F}_q[x]$, and the place at infinty $P_{\infty}$, identified with $1/x$.  Every place $P$ has a corresponding discrete valuation on $\mathbb{F}_q(x)$ denoted $v_P$ and a discrete valuation ring $\mathcal{O}_P = \{G \in \mathbb{F}_q(x) | v_P(G) \geq 0 \}$.  These definitions may be naturally extended to the field $\mathcal{F}$.  That is, the finite places are associated with the non-zero prime ideals in $\mathcal{O}_\mathcal{F}$ and the infinite places are associated to the non-zero prime ideals in the integral closure of $\mathcal{O}_{P_{\infty}}$.  If $\mathfrak{p}$ is a place of $\mathcal{F}$ then let $v_\mathfrak{p}$ denote its associated discrete valuation and $\mathcal{O}_\mathfrak{p} = \{ \alpha \in \mathcal{F} | v_\mathfrak{p}(\alpha) \geq 0 \}$ its discrete valuation ring.  There exists a place $P \in \mathbb{F}_q(x)$ with $v_\mathfrak{p}(P) > 0$; we say $\mathfrak{p}$ lies above $P$ and write $\mathfrak{p} | P$.  The positive integer $e(\mathfrak{p}|P) = v_\mathfrak{p}(P)$ is the ramification index and we say $P$ is ramified if $e(\mathfrak{p}|P) > 1$ and unramified otherwise.  Further, if $\gcd (e(\mathfrak{p}|P), q) = 1$ a place is called tamely ramified and wildly ramified otherwise.    The inertial degree of a place is denoted $f(\mathfrak{p}|P)$ and has value  $[ \mathcal{O}_\mathfrak{p}/\mathfrak{p} : \mathcal{O}_P/(P)]$ if $P$ is a finite place and $[\mathcal{O}_\mathfrak{p}/\mathfrak{p}:\mathbb{F}_q]$ for the infinite place.  

Knowing the splitting behavior of places is a key component to determine the genus of the function field $\mathcal{F}$.  We now turn our attention to characterizing the splitting behavior of all the places, starting with the finite places and concluding with the infinite place.  

\begin{theorem}\label{t:n_finiteplaces}
Let $P \in \mathbb{F}_q[x]$ be an irreducible polynomial and let $q_1 = q^{deg(P)}$.  Also let $a$ and $b$ be defined by  $T^3 - aT + b  \equiv T^3 - AT + B \pmod{P}$. Then the principal ideal $(P)$ splits into prime ideals in $\mathcal{O}_\mathcal{F}$ as follows:
\begin{enumerate}
\item If $v_P(\Delta) > 2 $ then $(P) = \mathfrak{p}^3$.
\item If $v_P(\Delta) = 1 $ then $(P) = \mathfrak{qp}^2$.
\item Otherwise $P\nmid A$, $d = gcd( T^{q_1} - T,  T^3 - aT + b)$, and we consider three cases:
\begin{enumerate}
\item If $\deg d = 0$ then $(P) = \mathfrak{p}$. 
\item If $\deg d = 1$ then $(P) = \mathfrak{pq}$.
\item If $\deg d = 3$ then $(P) = \mathfrak{pqr}$. 
\end{enumerate}
\end{enumerate}
\end{theorem}

\begin{proof}
For primes not dividing $A$, $\{1, y, y^2\}$ is an integral basis of $\mathcal{O}_P[y]/\mathcal{O}_P$ and thus Kummer's Theorem (see Theorem III.3.7 of \cite{sticht}) may be applied to get the desired result.   By Dedekind's Different Theorem, ramified primes are distinguished by the multiplicity with which they divide the field discriminant and thus the the two ramified cases are as claimed  (see Theorem III.5.1 in \cite{sticht}). 
\end{proof}

While we could consider a transformation to bring the infinite place to a finite place and invoke Kummer's Theorem as above, there is no guarantee that the infinite place is nonsingular.  We will avoid this approach and appeal to completions using Theorem 3.1 of \cite{exp_cff}.  This theorem implies that there will be a root in $\mathbb{F} \langle x^{-1} \rangle$, where $\mathbb{F}$ is some finite extension of $\mathbb{F}_q$, if and only if the infinite place is not wildly ramified.  We will show that a curve in the form of (\ref{tame}) characterizes the infinite place being tamely ramified or unramified by constructing just such a root in $\mathbb{F} \langle x^{-1} \rangle$.  From the construction, it will then be a matter of  counting the number of roots, and hence finding $[\mathbb{F}: \mathbb{F}_q]$ as this corresponds to the inertial degree.   After dealing with the four cases that arise in this situation, it will be clear why (\ref{wildram}) implies that the place at infinity is wildly ramified.   

Assume the curve is in standard form and satisfies (\ref{tame}).  Consider constructing a root $y \in \mathbb{F} \langle x^{-1} \rangle$ of $\phi(T)$.  We can write 
\begin{equation*} y = y_nx^n + y_{n-1}x^{n-1} + \ldots \end{equation*}
where $y_i \in \mathbb{F}$.  Let $A(x) = a_{2n}x^{2n} + \ldots + a_0$ and $B(x) = b_{3n}x^{3n} + \ldots b_0$ with $a_i, b_i \in \mathbb{F}_q$.  By writing the polynomials this way, we only assume that either $a_{2n}$ or $a_{2n-1}$ is nonzero.  If $a_{2n}=0$ then $b_{3n} = 0$ and $b_{3n-1}= 0$ in order to satisfy (\ref{tame}).  The coefficients of the powers of $x$ in the equation $y^3 - A(x)y + B(x) = 0$ are as follows:

\begin{eqnarray*}
x^{3n}  \   : & \  y_n^3 - a_{2n}y_n + b_{3n} \\
x^{3n-1} \  : &\ -a_{2n-1}y_n - a_{2n}y_{n-1} + b_{3n-1}\\
x^{3n-2} \  : &\-a_{2n-2}y_n - a_{2n-1}y_{n-1} - a_{2n}y_{n-2} + b_{3n-2}\\
x^{3n-3} \  : &\ y_{n-1}^3 -a_{2n-3}y_n - a_{2n-2}y_{n-1} - a_{2n-1}y_{n-2} -a_{2n}y_{n-3} + b_{3n-3}\\
\vdots \quad \  : & \ \vdots  
\end{eqnarray*}

The equation associated with $x^{3n}$ is cubic in $y_n$.  After the initial cubic equation,  we have an equation associated to $x^{3n-i}$ that is linear in $y_{n-i}$ for $i > 0$. That is, the values for $y_{n-i}$ are uniquely determined by the initial choice for $y_n$.  Therefore, we examine the solutions to $Y^3 - a_{2n}Y + b_{3n} = 0$.  

If $\deg(A)$ is odd then $a_{2n} = 0$ and $Y^3 + b_{3n} = 0$ has exactly one solution.  This occurs in the partially ramified case because there is exactly one root in $\mathbb{F}_q \langle x^{-1} \rangle$ and another distinct root in $\mathbb{F}_q \langle x^{-1/2} \rangle$.  Otherwise when $a_{2n} \neq 0$, the finite field extension which contains $y_n$ determines the splitting type.  The completely split case has three distinct roots in $\mathbb{F}_q \langle x^{-1} \rangle$. In this case  $Y^3 - a_{2n}Y + b_{3n}$ splits completely in $\mathbb{F}_q[Y]$.   The inert case will have a root in $\mathbb{F}_{q^3} \langle x^{-1} \rangle$ and  therefore $Y^3 - a_{2n}Y + b_{3n}$ is irreducible over $\mathbb{F}_q[Y]$.  The last case corresponds to the polynomial factoring as a linear and an irreducible quadratic in $\mathbb{F}_q[Y]$.   Therefore the place is partially split.   

Suppose the curve in standard form satisfies (\ref{wildram}).  If we let $b_k$ be the leading coefficient of $B(x)$ (with $3 \not | k$, and $k>3n$),  then the first equation is $b_k = 0$ which is a contradiction.  This forces the above system of equations to be inconsistent. Therefore a curve in standard form satisfying (\ref{wildram}) must have a wildly ramified place at infinity.  We summarize the above discussion in following theorem. 

\begin{theorem}\label{infinity_splitting}
The place at infinity splits as follows.  
\begin{enumerate}
\item If $\phi(T)$ satisfies (\ref{wildram}) then $(\infty) = \mathfrak{p}^3$.
\item If $\phi(T)$ satisfies (\ref{tame}) and $\deg(A)$ is odd then $(\infty) = \mathfrak{p}\mathfrak{q}^2$.
\item If $\phi(T)$ satisfies (\ref{tame}) and $\deg(A)$ is even then $ d = {\rm gcd}(T^q - T, T^3 - a_{2n}T + b_{3n})$ determines the splitting type.  
\begin{enumerate}
   \item If $\deg d = 0$ then  $(\infty) = \mathfrak{p}$.
   \item If $\deg d = 1$ then  $(\infty) = \mathfrak{p}\mathfrak{q}$.
   \item If $\deg d = 3$ then  $(\infty) = \mathfrak{p}\mathfrak{q}\mathfrak{r}$.
\end{enumerate}
\end{enumerate}
\end{theorem}
 
This gives a classification of the signature of the function field that involves a simple computation using very basic parameters of the curve.  We can now turn our attention to calculating the genus of the function field.


\section{Genus}

We will calculate the genus with the Hurwitz Genus Formula (see Theorem III.4.12 in \cite{sticht}), which requires knowledge of the degree of the different.  Having the field discriminant, Dedekind's Different Theorem gives the different exponents for the finite places.  The infinite place is more complicated.  However, since the problematic case is when wild ramifiation occurs, we are fortunate in that wild ramification is also total ramification.  For totally ramified places, determining the different exponent is a matter of finding a uniformizer for the place and evaluating a particular valuation (see Theorem III.5.12 of \cite{sticht}).

\begin{lemma}
If the place at infinity is totally ramified then it has different exponent $\delta_{\infty} = 2\deg B  - 3\deg A + 2$.
\end{lemma}

\begin{proof}
Let $\mathfrak{p}$ be the place at infinity in $\F$.  Since it is totally ramified and must lie above the unique infinite place in $\Fq(x)$ with uniformizer $1/x$, $v_{\mathfrak{p}}(x) = -3$.  By examining the equation $y^3 - A(x)y + B(x) = 0$, we can determine $v_{\mathfrak{p}}(y) = -\deg B$.  We will apply Theorem III.5.12 of \cite{sticht}.  The uniformizer used in the theorem will depend on which residue class $\deg B$ resides in modulo $3$.  

If $\deg B = 3m - 1$ then a uniformizer of ${\mathfrak{p}}$ is given by $t = y/x^m$. The minimal polynomial for $t$ is $f(t) = t^3 - Atx^{-2m} + Bx^{-3m}$.   Applying the theorem we see
\begin{equation*} \delta_{\infty} = v_{\mathfrak{p}}(f'(t)) = v_{\mathfrak{p}}(Ax^{-2m}) = -3\deg A + 6m = 2\deg B -3\deg A + 2. \end{equation*}
The case  $\deg B = 3m + 1$ follows in a similar manner.
\end{proof}

\begin{theorem}
If $\mathcal{F}$ has a totally ramified place at infinity then the genus of $\mathcal{F}$ is $g = \deg B - \deg I - 1.$
\end{theorem}
\begin{proof}
The Hurwitz Genus Formula gives
\begin{equation*} 2g - 2 = -2[\mathcal{F}:\mathbb{F}_q(x)] + \sum_{\mathfrak{p} \in \mathbb{P}_{\mathcal{F}} } d(\mathfrak{p}|P). \end{equation*}
This yields
\begin{equation*}2g - 2 = -6 + (3\deg A - 2 \deg I ) + (2\deg B - 3 \deg A + 2), \end{equation*}
which upon simplification gives the desired result.
\end{proof}

\begin{theorem} If $\mathcal{F}$ has a place at infinity that is tamely ramified or unramified, then $g = (3\deg{A} - 2\deg I + \delta_{\infty} - 4 )/2$ where $\delta_{\infty} = 0$ if $\deg A$ is even and $\delta_{\infty} = 1$ if $\deg A$ is odd.
\end{theorem}
\begin{proof}
The proof follows as before except now that the infinite place is not wildly ramified, and hence its different exponent $\delta_{\infty}$ can only take the values 0 or 1.  Since the genus is an integer, the parity of $\deg A$ determines the value of $\delta_{\infty}$.
\end{proof}

We have described the basic invariants of cubic function fields in characteristic three.  The focal point for the remainder of this paper is to develop the arithmetic of ideals.  As in the previous sections, one can appeal to generic algorithms to solve this problem.  However, these algorithms typically require operations on large matrices or an appeal to Groebner basis.  We desire, like elliptic curves and hyperelliptic curves (using Cantor's algorithm), a method to do computations that depends only on the underlying curve parameters and the finite field.


\section{Divisor Class Groups and Ideal Class Groups}
This section provides an overview of the relationship between the Jacobian of a curve and the ideal class group of a function field.  As there are many sources for this material (see e.g. \cite{bauer}, \cite{landquist}, \cite{hasse} \cite{aacff}, \cite{idealarithmetic}), we will be relatively brief and only provide the relevant definition and results where needed. Once this is completed, it will be possible to develop arithmetic on ideals and, for a certain class of curves, fully realize arithmetic in the ideal class group.  

A divisor is a finite formal sum of places in $\mathcal{F}$.  The set of all divisors forms a free abelian group.  We will work in a specific finite subgroup of this group.  Let $S$ be the set of finite places in $\mathcal{F}$.  There is an isomorphism between the divisors with support in $S$, $\mathcal{D}_F(S)$, and the fractional ideals in $\mathcal{O}_\mathcal{F}$, $\mathcal{I}(\mathcal{O}_\mathcal{F})$.  The \emph{Fundamental theorem of ideal theory in an algebraic function field} \cite[p 401]{hasse} gives the isomorphism as
\begin{equation}\label{fund_iso} \Phi \ : \ \mathcal{D}^0_\mathcal{F}(S) \rightarrow \mathcal{I}(\mathcal{O}_\mathcal{F}), \quad D \longmapsto \left\{ \alpha \in F^{\times} \ \vrule \ \sum_{P \in S}v_P(\alpha)P \geq D \right\} \cup \{ 0 \}. \end{equation}

This may also be defined by 
\begin{equation*} \sum n_PP \longmapsto \prod_{P \in S} (P \cap \mathcal{O}_\mathcal{F})^{n_P}. \end{equation*}
In general the ideal class group is related to the Jacobian by the following exact sequence (see Theorem 14.1 of \cite{rosen})
\begin{equation*} (0) \rightarrow \mathcal{D}_\mathcal{F}(S^c)/\mathcal{P}_\mathcal{F}(S^c) \rightarrow \mathcal{J}_\mathcal{F} \rightarrow \mathcal{C}l(\mathcal{O}_F) \rightarrow \mathbb{Z}/f\mathbb{Z} \rightarrow (0), \end{equation*}
where $S^c$ is the set of infinite places (the set compliment of $S$ in $\mathbb{P}_\mathcal{F}$).  Specifically, if a function field has a unique place at infinity of degree 1, the points on the Jacobian will be isomorphic to the ideal class group.  

We use the hierarchy of divisors (and hence ideals) defined in \cite{bauer} so that there is a way to represent elements of the divisor class group of degree zero in a unique way with minimal information.  A divisor $D$ is \emph{effective} if $D>0$ (that is, $n_P \geq 0$ for all $P \in \mathbb{P}_F$)  and denote its effective part as $D^+$, i.e.
\begin{equation*} D= \sum_{P\in \mathbb{P}_\mathcal{F}} n_PP  \quad  \Rightarrow \quad  D^+ = \sum_{P\in \mathbb{P}_\mathcal{F}, n_P > 0} n_PP .\end{equation*}
A degree zero divisor is called \emph{finitely effective} if its finite part is effective; it can be shown that every divisor $D \in \mathcal{D}_\mathcal{F}^0$ is equivalent to a finitely effective divisor.  This is the first step in the hierarchy.

A finitely effective divisor is \emph{semi-reduced} if there does not exist a non-empty sub-sum of the form $(\alpha)$ where $\alpha \in \mathbb{F}_q[x] \backslash \mathbb{F}_q$. Again, it is straightforward to show that every divisor is equivalent to a semi-reduced divisor, extending the hierarchy.  A \emph{semi-reduced} divisor $D$ is \emph{reduced} if $\deg D^+ \leq g$ where $g$ is the genus of the curve.  Using the Riemann-Roch Theorem, it is possible to prove that every divisor class also contains a reduced divisor.

To complete the hierarchy, we define a \emph{distinguished} divisor  to be a divisor $D$ such that for all other equivalent finitely effective divisors $D_1$, we have that $\deg D_1^+ \leq \deg D^+$ implies $D = D_1$.   If a divisor is distinguished,  it is reduced \cite[Lemma 1.12]{bauer}.  Unfortunately, we have no apriori way of knowing if such a divisor exists or of verifying that a divisor is distinguished.  

The above definitions for a divisor $D$ can immediately be transferred to fractional ideals by first considering $D^+$ and then applying the isomorphism \eqref{fund_iso}.  Finitely effective divisors map to integral ideals, and hence we can do computations in this context.  Note that in the ideal class group we will mostly work with primitive ideals, that is: $\mathfrak{a}$ is primitive if and only if there is no non-constant polynomial $a(x) \in \mathbb{F}_q[x]$ such that $\langle a(x) \rangle \ | \ \mathfrak{a}$ where $\langle a(x) \rangle$ represents $a(x) \mathcal{O}_F$.  Under the above correspondence, we see that primitive integral ideals give an equivalent notion to semi-reduced divisors.  We will call an ideal reduced (resp. distinguished) if it is the image under the above correspondence of a reduced (resp. distinguished) divisor.  We now turn our attention to determining when it is possible to show that each divisor class, or equivalently, ideal class, contains a distinguished element.

 Let $\alpha = a + b\rho + c\omega \in \mathcal{F}$ with $a,b,c \in \mathbb{F}_q(x)$.  Then the norm of $\alpha$ is given by
\begin{equation*} N_{\mathcal{F}/\mathbb{F}_q(x)}(\alpha) = N(a + b\rho + c\omega)= \end{equation*}
\begin{equation*} a^3 - a^2cE + abcIF -ab^2A + b^2cAE + bc^2AF - bc^2EFI -c^3F^2I -b^3FI^2. \end{equation*}

\begin{theorem}\label{t:n_norm}
Let $\alpha = a + b\rho + c\omega \in \mathcal{O}_\mathcal{F}$, $2 \deg B > 3 \deg A$,  and $3\nmid \deg FI^2$ .  Then $ \deg{N(\alpha)} = \max\{\deg{a^3}, \deg{b^3FI^2}, \deg{c^3F^2I}\}$. 
\end{theorem}

The proof follows from a careful analysis of the degrees of the relative terms in the norm expression, and noting that the criterion that $3\nmid \deg FI^2$ actually forces $\deg FI^2 = \deg B$ and thus the curve satisfies \eqref{wildram}.  A detailed version of the proof can be found in \cite{webster}.

It is natural to wonder if \eqref{wildram} implies $3 \nmid \deg{FI^2}$.  Unfortunately, a class of curves exists for which this implication is not true.  We can construct a curve such that $3\nmid \deg{B}$ and $3|\deg{FI^2}$.  In general we do not expect to deal with such curves; it requires a very special sort of singularity.  An example of this type of singularity is given in the following construction.

\begin{example}
Consider the function field given with parameters $A =  (x^2 + x - 1)(x^2 + 1)$ and $B = - x^8 + x^6 + x^5 + x^4 + x^2 + 1$.  These parameters define a curve that is in standard form and satisfies \eqref{wildram}. Both divisors of $A$ are singular, $I = (x^2 + x - 1)(x^2 + 1)$, and $i = x^3 + x^2$.  Thus $\deg{FI^2} =\deg{( i^3 - iA + B)} = 9$.  
\end{example}

Having established this property of the norm, we can now return to the specifics of distinguished ideals.  In particular, Theorem \ref{t:n_norm} is exactly what is needed to extend Theorem 5.1 of \cite{bauer} to this case.

\begin{theorem}\label{t:minnormexists}
If $2 \deg B > 3 \deg A$,  and $3\nmid \deg FI^2$, then every nonzero ideal contains a nonzero element of minimal norm which is unique up to multiplication by an element in $\mathbb{F}_q^{\times}$.
\end{theorem}

The proof is identical to that in \cite{bauer} but we sketch the key points.  The validity is established using Theorem \ref{t:n_norm}.  Assume there are two elements $\alpha_i = a_i + b_i\rho + c_i\omega$ for $i = 1,2$ whose norm has the same degree and suppose $\deg N(\alpha_i) = \deg a_i^3$.  Let $k$ be the quotient of the leading term of $a_1$ divided by the leading term of $a_2$ then $\alpha_3 = \alpha_1 - k\alpha_2$ has smaller norm.   A similar argument works when the degree of the norm is determined by $b_i$ or $c_i$.  

\begin{theorem}\label{distinguished}
If $2 \deg B > 3 \deg A$,  and $3\nmid \deg FI^2$, then every ideal class contains a unique distinguished ideal.
\end{theorem}

\begin{proof}

By Theorem \ref{t:minnormexists} there exists $\alpha_1 \in J$ of minimal norm from some ideal $J$.  We consider the primitive integral ideal $J_1 = \langle \alpha_1 \rangle J^{-1}$ .  Assume there is some integral primitive ideal $J_2$ equivalent to $J_1$ with $\deg N(J_2) \leq \deg N(J_1)$.  Then $J_2J = \langle \alpha_2 \rangle$ with $\alpha_2 \in J$.  This gives $\deg N(J_1) = \deg N(\alpha_1) - \deg N(J) \geq \deg N(J_2) = \deg N(\alpha_2) - \deg N(J)$ which implies $\deg N(\alpha_1) \geq \deg N(\alpha_2)$.  By assumption $\alpha_1$ is an element of minimal norm; therefore $\alpha_2 = k\alpha_1$ for $k \in \mathbb{F}_q^{\times}$ and $J_1 = J_2$.  Thus every ideal class contains a unique distinguished ideal. 
\end{proof}

All the theoretical pieces are in place to develop arithmetic in the ideal class group.  Ideal inversion and multiplication pose no major theoretical obstacles,  and the above establishes a unique way to find a distinguished ideal in a given class.  Combining all of the pieces will allow composition and reduction in the ideal class group.  The remaining sections make the above explicit for the considered function fields.  We start by describing an integral basis for primes and their powers which gives insight to how inversion and multiplication will work as explained afterward.  Finally, we give explicit algorithms that produce the element of minimal norm and its corresponding distinguished ideal.


\section{Triangular basis for prime ideals}

Having described how the finite places split, it will be helpful to have a concrete description of generators for the prime ideals in terms of the basis elements developed in Section 4.  Scheidler provided a comparable statement in Theorem 3.1 of \cite{idealarithmetic} for all prime ideals in a purely cubic function field of characteristic not 3 that was an analog of the theorem of Voronoi \cite{vor} for number fields.   Having classified the splitting type of prime ideals, we follow their lead and give the triangulr bases along with basic products and powers of the prime ideals.  

Throughout the following sections, proofs will occasionally be omitted for the sake of brevity.  In particular, when a particular technique may be used successfully to compute the basis in multiple cases, it will only be included once.  The interested reader may always refer to \cite{webster} for complete proofs.


\subsection{Ramified primes}

There are three cases to consider for the ramified primes.  When calculating powers of primes, ramification tends to make the treatment here a little easier for a given prime.  A ramified prime is expected to be totally ramified so $\mathfrak{p}^3 = (P)\mathcal{O}_\mathcal{F} = (P)[1, \rho, \omega]$ for some irreducible polynomial $P \in \Fq[x]$.  This leaves only the calculation of the basis for $\mathfrak{p}$ and  $\mathfrak{p}^2$.  

\begin{prop}\label{t:wildrambasis}
Let $v_P(A) \geq 1$ and $v_P(I) = 0$ so that $(P) = \mathfrak{p^3}$.  Then 
\begin{equation*} \mathfrak{p} = [P, f + \rho,- I^{-1}f^2 + \omega] 
\mbox{ \ and \ }
\mathfrak{p}^2 = [P, P\rho, I^{-1}f^2  - I^{-1}f\rho +  \omega]\end{equation*} 
where $f^3 \equiv FI^2 \pmod{P}$,  and $I^{-1}I \equiv 1 \pmod{P}$. 
\end{prop}

\begin{proof}
Apply Kummer's theorem to the minimal polynomial of $\rho$ to see 
\begin{equation*} \rho^3 -A\rho + FI^2 \equiv \rho^3 + FI^2 \equiv (\rho + f)^3 \pmod{P}.\end{equation*}
This implies $\mathfrak{p} = \langle P, f + \rho \rangle$.  To get the last element  we consider $\rho (f + \rho)  = f\rho + I \omega + A \in \mathfrak{p}$. Using that $I$ is relatively prime to $P$ we get the last element as claimed.   

For $\mathfrak{p}^2$, consider $(f + \rho)^2 = f^2 - f\rho + I\omega + A \in \mathfrak{p}^2$ to see that the third term in the basis has the form claimed.  Note $v_{\mathfrak{p}}(P) = 3$, so $P \in \mathfrak{p}^2$.  Since the ideal has to have norm $P^2$ this forces the second element of the basis to be as stated.    
\end{proof}

The next two primes we consider both lie over primes dividing the index.  Unlike the above primes and the unramified primes (as we will see in the next subsection), these primes can give rise to ideals that will have $\omega$ with a polynomial coefficient.  

\begin{prop}\label{sing_wildram}
Let $v_P(A) > 1$ and $v_P(I) = 1$ so that  $(P) = \mathfrak{p}^3$.  Then
\begin{equation*}\mathfrak{p} = [P, \rho, \omega]
\mbox{ \ and \ }
\mathfrak{p}^2 = [P, \rho, P\omega]. \end{equation*}
\end{prop}

\begin{proof}
By applying Kummer's theorem to the minimal polynomials of $\rho$ and $\omega$, we see that both are in $\mathfrak{p}$.  Squaring $\mathfrak{p}$ and considering the 9 resulting elements leads to the latter result.
\end{proof}

Tamely ramified primes are all that remain to consider.  We state their bases here, but we will treat powers of $\mathfrak{p}$ and $\mathfrak{q}^2$ with unramified primes.

\begin{prop}\label{splitramified}\label{splitram}
Let $v_P(A) = 1$ and $v_P(I) = 1$ so that  $(P) = \mathfrak{pq}^2$.  Then
\begin{equation*} \mathfrak{p} = [P, \rho, E +  \omega], \quad \mathfrak{q} = [P, \rho, \omega], \end{equation*}
\begin{equation*} \mathfrak{q}^2 = [P, P\rho, E^{-1}F\rho + \omega],\mbox{ and }\ \mathfrak{pq} = [P, \rho, P\omega] \end{equation*}
where $E^{-1}$ is the inverse of $E$ modulo $P$.  
\end{prop}

\begin{proof}
Apply Kummer's theorem to the minimal polynomials of $\rho$ and $\omega$ and proceed as before.
\end{proof}


\subsection{Unramified primes}

The unramified primes are the primes that we expect to deal with most often in the course of doing computations.  There are several different cases to be accounted for depending on the inertial degree and the primes involved in the product.  First we deal with the case in which $\mathfrak{p}$ has inertia degree 1, and then when $\mathfrak{p}$ has inertia degree 2.  This case will also handle the product of two primes of inertia degree 1 and the powers of $\mathfrak{q}^2$ for the split ramified case.  The last case will be where  $\mathfrak{p}$,  $\mathfrak{q}$ both have inertial degree 1 and we need to consider products of the form $\mathfrak{p}^i\mathfrak{q}^j$ with $i \neq j$.

\begin{prop}\label{unramified}
  Let $\mathfrak{p}|P$ have inertial degree 1 and ramification index 1. Then 
  \begin{equation*} \mathfrak{p} = [ P, - \alpha + \rho, - I^{-1}(\alpha^2 - A) + \omega] \end{equation*}
   where $\alpha$ is a root of the minimal polynomial of $\rho$ modulo $P$ and $I^{-1}I \equiv 1 \pmod{P}$.
\end{prop}

\begin{proof}  The ideal $\mathfrak{p}$ is generated by $\langle P, -\alpha + \rho \rangle$, and the rest follows.
\end{proof}

The proposition below deals with primes that have ramification index 1 and inertia degree 1.  It therefore also handles the unramified prime lying over the tamely ramified primes.  

\begin{prop}\label{t:n_single}
For $\mathfrak{p}$  with ramification index 1 and inertial degree 1, we have 
\begin{equation*} \mathfrak{p}^i= [P^i, -X_i + \rho, -Z_i +  \omega] \end{equation*}
 where 
 \begin{itemize}
 \item $Z_{i+1} = Z_i +  kP^i$,  
 \item$k \equiv -C_i(EZ_i)^{-1} \pmod{P}$, 
 \item $C_i = -(Z_i^3 - EZ_i^2 + F^2I)/P^i$,  
 \item $X_{i+1} \equiv -FIZ_{i+1}^{-1} \pmod{P}$, 
\end{itemize}
and $X_1$ and $Z_1$ are defined and given in Propositions \ref{splitramified} and \ref{unramified}. 

\end{prop}
\begin{proof}

The definitions in this proposition make it important that $Z_1$ be invertible modulo $P$.  In Proposition \ref{splitramified}, $E$ is invertible modulo $P$.  For Proposition \ref{unramified} the element $Z_1$ is invertible because it is a  nonzero root of the minimal polynomial of $\omega$ modulo $P$, that is to say only ramified primes correspond to $0$ being a root modulo $P$.   

Since $P^i|N(\omega - Z_i)$, that basis element can be written as $\omega - (Z_i + kP^i)$.  We now describe how to choose $k$ so that the element is correct for $\mathfrak{p}^{i+1}$.
\begin{align*}
N(- (Z_i + kP^i) + \omega ) & = -[(Z_i + kP^i)^3 + E(Z_i + kP^i)^2 +F^2I] \\
                        &\equiv -(Z_i^3 - EZ_i^2 + F^2I) - EZ_ikP^i \pmod{P^{i+1}}\\
                        &\equiv -(C_iP^i - EZ_ikP^i) \pmod{P^{i+1}}
\end{align*}
Since we want $C_iP^i - EZ_ikP^i \equiv 0 \pmod{P^{i+1}}$, we can choose $k \equiv C_i(EZ_i)^{-1} \pmod{P}$.  Such an inverse exists because $P$ is relatively prime to both $E$ and $Z_i$.  Now that $ - Z_{i+1} + \omega \in \mathfrak{p}^{i+1}$ we can see that  $(Z_{i+1} - \omega)\rho = FI +  Z_{i+1}\rho  \in \mathfrak{p}^{i+1}$.  This gives the term with $-X_{i+1} + \rho$ as claimed.
\end{proof}

\begin{prop}\label{partsplit}
Let $\mathfrak{q}$ be a prime with inertia degree 2.  Then 
\begin{equation*} \mathfrak{q} = [ P , P\rho, I^{-1}(W + A) - I^{-1}M\rho  + \omega] \end{equation*}
 where $\rho^3 - A\rho + FI^2 \equiv (\rho-\alpha)(\rho^2 - M\rho + W) \pmod{P}$.
\end{prop}

\begin{proof}
Kummer's theorem gives $\mathfrak{q} = \langle P, \rho^2 - M\rho + W \rangle$, and similar techniques complete the proof.
\end{proof}

Notice the form of the product of two distinct unramified primes lying over a completely split prime:
\begin{align*} 
\mathfrak{pq} &= \langle P,- \alpha_1 + \rho  \rangle \langle P, - \alpha_2 + \rho \rangle \\
              &= \langle P^2, P(-\alpha _ 1 + \rho), P(- \alpha_2 + \rho),  A +\alpha_1\alpha_2 -(\alpha_1 + \alpha_2)\rho +I\omega \rangle \\
              &= [ P, P\rho,  I^{-1}(A +\alpha_1\alpha_2)  -I^{-1}(\alpha_1 + \alpha_2)\rho + \omega ].
\end{align*}
Here the last line is justified by the fact that $(\alpha_1 - \alpha_2)$ is relatively prime to $P$. Thus, the greatest common divisor of $P(\alpha_1 - \alpha_2)$ and $P^2$ is $P$.   There are three types of ideals that can have the form $[P, P\rho, -N_1 - M_1\rho + \omega]$:
\begin{itemize}
\item $\mathfrak{q} = [P, P\rho, \omega - M\rho - W]$ from Proposition \ref{partsplit},
\item $\mathfrak{q}^2 = [P, P\rho, E^{-1}F\rho + \omega ]$ from Proposition \ref{splitramified}, and
\item $\mathfrak{pq} =[ P, P\rho,  I^{-1}(A +\alpha_1\alpha_2)-I^{-1}(\alpha_1 + \alpha_2)\rho + \omega]$ from the exposition above.
\end{itemize}

\begin{prop}\label{t:n_double}
Let $\mathfrak{r}$ represent any of the three ideals above. Then 
\begin{equation*} \mathfrak{r}^i = [P^i, P^i\rho, N_i - M_i\rho + \omega] \end{equation*}
where 
\begin{itemize}
\item $L(M_{i-1}M_1I + N_{i-1} + N_1 - E) \equiv 1 \pmod{P^i}$, 
\item $M_i \equiv -L(F + M_{i-1}N_1 + M_1N_{i-1}) \pmod{P^i}$, and
\item  $N_i \equiv L( M_1FI + M_{i-1}FI M_{i-1}M_1A N_{i-1}N_1) \pmod{P^i}$.   
\end{itemize}
\end{prop}

\begin{proof}
The previous work establishes the base case $i = 1$ and we argue by induction.

\begin{equation*} \mathfrak{r^{i-1}}\mathfrak{r} = [ P^{i-1}, P^{i-1}\rho, \omega - M_i\rho + N_i][P, P\rho, \omega - M_1\rho + N_1]. \end{equation*}

In the nine possible products of the basis elements only $ (\omega - M_i\rho + N_i)(\omega - M_1\rho + N_1)$ does not contain a factor of $P$.  Thus the coefficient of $\omega$ has to be relatively prime to $P$.  If it were not, the product would not be primitive. Multiplying through by its inverse modulo $P^i$ gives the desired basis element. 

The product contains $P^i$ and $P^i\rho$.  A norm argument shows that $\mathfrak{r}^i$ cannot contain  $P^{i-1}$ or $P^{i-1}\rho$.  Thus the ideal has the desired norm and the elements stated form a basis.
\end{proof}

All that remains is to handle $\mathfrak{p}^i\mathfrak{q}^{i + j}$ where $j>0$ and each prime has inertia degree~1.   By Proposition \ref{t:n_double} we know 
\begin{align}
\label{prod1} ( \mathfrak{pq})^i& = [ P^i, P^i\rho, N_i - M_i\rho + \omega] \\
\intertext{and by Proposition \ref{t:n_single}}
\label{prod2} \mathfrak{q}^j &= [P^j, - X_{\mathfrak{q}_j} + \rho , - Z_{\mathfrak{q}_j} + \omega ], \\
\label{prod3} \mathfrak{q}^{i+j} &= [P^{i+j}, - X_{\mathfrak{q}_{i+j}} + \rho,   - Z_{\mathfrak{q}_{i+j}} + \omega ], \mbox{\quad and } \\
\label{prod4} \mathfrak{p}^i &= [P^i,  - X_{\mathfrak{p}_i} + \rho, - Z_{\mathfrak{p}_i} + \omega ]. 
\end{align}
Combinations of the above products will help determine the proper basis of $\mathfrak{p}^i\mathfrak{q}^{i + j}$.

\begin{prop}\label{partial_split}
Using notation as above
\begin{equation*} \mathfrak{p}^i\mathfrak{q}^{i + j} = [ P^{i+j}, P^i( - X_{\mathfrak{q}_j}+ \rho ), H + G\rho + \omega  ] \end{equation*}
where we let $N$ be defined by $NX_{\mathfrak{p}_i} \equiv 1 \pmod{P^{i+j}}$ and
\begin{equation*} G \equiv NZ_{\mathfrak{q}_{i+j}} \pmod{P^i} \mbox{ \quad and \quad } H \equiv N(-FI - X_{\mathfrak{p}_i}Z_{\mathfrak{q}_{i+j}}) \pmod{P^{i+j}}. \end{equation*}
\end{prop}  

\begin{proof}
Considering the product of \eqref{prod1} and \eqref{prod2}, we see that $P^{i+j}$ and $P^i( - X_{\mathfrak{q}_j} + \rho )$ are in $\mathfrak{p}^i\mathfrak{q}^{i + j} $.  By considering the product of \eqref{prod3} and \eqref{prod4}, we can see that 
\begin{equation*} ( - X_{\mathfrak{p}_i} + \rho) (- Z_{\mathfrak{q}_{i+j}} + \omega ) \in  \mathfrak{p}^i\mathfrak{q}^{i + j} \end{equation*}
Since $X_{\mathfrak{p}_i}$ is relatively prime to $P$ it is invertible modulo $P^{i+j}$.  Multiplying through by its modular inverse gives the third element of the basis.   The other two elements are in the ideal by construction.  It remains to establish that they are indeed basis elements, which is easily accomplished by a norm argument.
\end{proof}

We have dealt with all of the prime ideals and their possible powers and products.  We now turn to arbitrary ideal arithmetic.  Any  given ideal factors into the product of four ideals:
\begin{equation*} J = [s_1, s_1'(u_1 + \rho), v_1 + w_1\rho + \omega][s_2,  s_2(u_2 + \rho), v_2 + w_2\rho + \omega] \end{equation*}
\begin{equation*} [s_3, \rho, s_3''\omega][s_4, s_4'(u_4 + \rho), s_4''(v_4 + w_4\rho + \omega)] = J_1J_2J_3J_4. \end{equation*}
Each of the ideals is relatively prime to the others and is determined by which type of prime appears in the factorization.  For $\mathfrak{p}|P$, we have four criteria:
\begin{itemize} 
\item $\mathfrak{p}$ divides $J_1$ iff $P$ is unramified,
\item $\mathfrak{p}$ divides $J_2$ iff $P$ is totally ramified and does not divide the index,
\item $\mathfrak{p}$ divides $J_3$ iff $P$ is totally ramified and divides the index, and
\item $\mathfrak{p}$ divides $J_4$ if and only if $P$ is split ramified.  
\end{itemize}
We call these primes Type I, Type II, Type III and Type IV, respectively.  Recombining ideals factored in this way is a straightforward  application of the Chinese Remainder Theorem, while finding the factorization for a given ideal is an application of polynomial factorization.  There are a few reasons for this approach.  The first is for simplicity as the propositions are easier to state for a given type.  The combined proposition would look much like each constituent part where the final result is an application of the Chinese Remainder Theorem.  The second reason is that the difficulty often lies in a particular case and this allows the exposition to highlight the trouble.  

Furthermore, from a computational perspective, we are also drawn to this approach.  Two of the four cases involve curves that have singularities, and hence we can choose to avoid them.  We could also easily choose a curve with no finite ramification and ignore three of the four cases.  Even in the worst case scenario where all types of primes are possible, we still do not expect to deal with three of the four products in the course of doing arithmetic.  A rough heuristic argument shows that probability of two randomly chosen ideals with degree less than $g$ contain a ramified prime is $4g/q$, which will be small if $q$ is large.  Thus from a computational point of view, three of the four cases will rarely occur even when a curve is singular.


\section{Inversion and Division}

Some basic properties of the structure of ideals in cubic function fields developed in \cite{idealarithmetic} remain true even in characteristic three.  We cite without proof the containment criterion for ideals written with a triangular basis. 

\begin{prop}{\bf (Lemma 4.1 of \cite{idealarithmetic})}
Let $I_i = [s_i, s_i'(u_i + \rho), s_i''(v_i + w_i\rho + \omega)]$ for $i = 1,2$ be two ideals.  Then $I_1 \subseteq I_2$ if and only if 
\begin{equation*} s_2 | s_1, \quad s_2' | s_1', \quad s_2''|s_1'', \quad s_1'u_1 \equiv s_1'u_2 \pmod{s_2},\end{equation*}
\begin{equation*} s_1''w_1     \equiv s_1''w_2 \pmod{s_2'},\mbox{ \ and \ } s_1''v_1     \equiv s_1''(v_2 + u_2(w_1-w_2)) \pmod{s_2}.\end{equation*}
\end{prop}

Our first goal is to develop ideal inversion.  As we only wish to work with integral ideals, we compute a primitive ideal that is in the ideal class of the inverse of a given ideal.   As a reminder, the notation for such an inverse will be $\overline{J}$ and the notation for division will be $J^{-1}$.  We label the propositions depending on which of the four ideal types the proposition covers.

\begin{prop}\label{t:inverse_s_1}{\bf (Inversion for Type I and II primes) }
If $I_1 = [s, s'(u + \rho), v + w\rho + \omega]$, then $I_2 = \overline{I_1} =  \langle s \rangle I_1^{-1}$ is given by $I_2 = [S, S'(U+ \rho), V + W\rho +\omega]$, where
\begin{equation*} S = s, \quad S' = s/s', \quad U \equiv -Iw  \pmod{s'},\end{equation*} 
\begin{equation*} W \equiv -uI^{-1} \pmod{s/s'}, \ \mbox{ and } \ V \equiv  E - v - WIw \pmod{s}. \end{equation*} 
\end{prop}

\begin{proof}
Since $s \in I_1$, it is clear that $\langle s \rangle I_1^{-1}$ is an integral ideal.  We show that the above choices provide a correct $\Fq[x]$ basis for $I_2$.  The fact that $I_1I_2 = \langle s \rangle$ will be used extensively in this proof (and the proofs to follow).  Since $s \in I_2$, $S|s$.  Examining $S(v + w\rho + \omega) \in \langle s \rangle $,  we conclude $s|S$ and hence $s = S$. Consider the norm of the ideal $\langle s \rangle$ to determine $S'$:
\begin{equation*}s^3 = N( \langle s \rangle ) =   N(I_1)N(I_2) = ss'sS'. \end{equation*}
Therefore $S' = s/s'$ as claimed.  The other products of the two ideals give the remaining congruences.  We start with $S'(U + \rho)(v + w\rho + \omega)$ and examine the coefficient of $\omega$:
\begin{equation*}S'(U + \rho)(v + w\rho + \omega) \in \langle s \rangle \Rightarrow s \ \vrule \ \frac{s}{s'}(U + Iw) \Rightarrow U \equiv -Iw \pmod{s'}.\end{equation*} 
The congruence for $W$ (resp. $V$) follows by considering the coefficient of $\omega$ in the product $s'(u + \rho)(V + W\rho + \omega)$   (resp. $(V + W\rho + \omega)(v + w\rho + \omega)$  ) and arguing as above. 
\end{proof}

We note that the above proposition is simpler for nonsingular curves because $I = 1$ and most of the congruences can be replaced by equalities.  It was fortunate that we could deal with Type I and Type II primes with a single proposition without an appeal to the Chinese Remainder Theorem.  We will often be able to deal with these two types of primes together.  

\begin{prop}\label{t:inverse_s_3}{\bf (Inversion for Type III primes) }
If $I_1 = [s,  \rho, s''\omega]$, then $I_2 = \overline{I_1} =  \langle s \rangle I_1^{-1}$ is given by $I_2 = [s, \rho, (s/s'')\omega]$.
\end{prop}

\begin{proof}
This follows immediately from Proposition \ref{sing_wildram}.
\end{proof}

Here the index divisors do not seem too complicated, but for Type IV ideals they will prove very troublesome.  The approach taken for this case will become familiar.  Just as factorization helps simplify the four cases, factorization within this case will prove useful.   Our approach is to consider a factorization that most closely resembles some of the basic prime powers of Section 7.  The treatment of Proposition \ref{partial_split}, that dealt with products of the form $\mathfrak{p}^i\mathfrak{q}^{i+j}$,  is enlightening here.  In the proof we appeal to the fact that the product can be viewed as $(\mathfrak{pq})^i$ and $\mathfrak{q}^j$.  This is the sort of factorization that we will use in many of the following propositions.  

\begin{prop}\label{t:inverse_s_4}{\bf (Inversion for Type IV)} 
If $I_1 = [s, s'(u + \rho), s''(v + w\rho + \omega)]$, then $I_2 = \overline{I_1} =  \langle s \rangle I_1^{-1}$ is given by $I_2 = [S, S'(U+ \rho), S''(V + W\rho +\omega)]$, where
\begin{equation*} S = s, \quad S' = \frac{s}{s's''s_I}, \quad S'' = s_I =  \gcd\left( \frac{s}{s's''} , v \right) \end{equation*}
\begin{equation*}  U \equiv \left\{ \begin{array}{ll} 
                    0  & \pmod{s''s_I} \\ 
                    -Iw & \pmod{s'} 
                    \end{array} 
             \right.  ,   \quad \quad
             V \equiv \left\{ \begin{array}{ll}
                     0 & \pmod{ s'' }\\
                     0 & \pmod {s/s's''s_I}\\
                     E  & \pmod{ s' }
                     \end{array}
             \right. ,
             \end{equation*}
\begin{equation*} W \equiv  E^{-1}F  \pmod{ s/s's''s_I}, \end{equation*}
and $s''$, $s/s's''s_I$, and $s'$ are pairwise coprime.
\end{prop}

\begin{proof}
Factor $I_1$ as
\begin{equation*}I_1 = [s'', \rho, s''\omega][s', s'\rho, w\rho + \omega] \left[ \frac{s}{s's''}, u + \rho, v + w\rho + \omega \right].\end{equation*}
By factoring the ideal in this fashion, we can find the inverse of each factor.  The inverse of the first two factors is an immediate consequence of Proposition \ref{splitram}. The inverse of the last ideal in the above factorization has two factors since it could contain either ramified primes or powers of unramified primes, which is determined by the term associated with $\omega$ and $s_I$.  For the ramified primes in this product, the inverse is $[s_I, \rho, s_I\omega]$ and this gives $V \equiv 0 \pmod{s_I}$.  The remaining factor of the inverse has the form
\begin{equation*} \left[ \frac{s}{s's''s_I}, \frac{s}{s's''s_I}\rho,  E^{-1}F\rho + \omega \right], \end{equation*}
yielding the only congruence for $W$ and the remaining congruence for $V$.   The above immediately shows that the choices for $S'$, and $S''$ are correct.  A quick norm argument shows that $S = s$ as claimed.  
\end{proof}

  The remaining portion of this section leads to arbitrary ideal division.  We begin with a series of lemmata that will handle the simplest case of division, and will later be used to handle the general case.  Consider  two primes $\mathfrak{p}$ and $ \mathfrak{q}$ lying over a completely split place $P$.  Given $ (\mathfrak{pq})^i$ and $\mathfrak{p}^i$, we find (somewhat trivially) $(\mathfrak{pq})^i\mathfrak{p}^{-i} = \mathfrak{q}^i$.  Thus the product $(\mathfrak{pq})^i$ is split into two ideals of equal norm, $\mathfrak{p}^i$ and $\mathfrak{q}^i$.  

\begin{lemma}\label{l:splitting_s_1}{\bf (Splitting for Type I and II primes)}
Let $I_2 = [s, s\rho, v_2 + w_2\rho + \omega]$ and $I_1 = [s, u_1 + \rho, v_1 + \omega ]$ be two ideals such that $I_2 \subseteq I_1$.  Then $J = I_2I_1^{-1} = [s, U + \rho, V + \omega]$, where 
\begin{equation*} U \equiv Iw_2 - u_1 \pmod{s}, \quad V \equiv v_2 - Iw_2^2 +u_1w_2 \pmod{s}. \end{equation*}  
\end{lemma}

\begin{proof}
By Proposition \ref{t:inverse_s_1}, $I_2 = \langle s \rangle [s, Iw_2 + \rho,  E - v_2 + \omega]^{-1}$.  Therefore we can write
\begin{equation*} J [s, u_1 + \rho, v_1 + \omega][s, Iw_2 +\rho, u_2w_2 - v_2 + \omega ] = \langle s \rangle. \end{equation*}
Since $J = [s, \rho + U, \omega + V]$, it is only a matter of finding the correct congruences for $V$ and $U$.  Using $(U + \rho)(u_1 + \rho)(Iw_2 +\rho) \in \langle s \rangle$ and the coefficient of $\omega$, we find $U \equiv Iw_2 - u_1 \pmod{s}$.  To find $V$, we note that $v_2 + w_2\rho + \omega \in J$ and subtract $w_2(U + \rho)$.
\end{proof}

\begin{lemma}\label{l:splitting_s_3}{\bf (Splitting for Type III primes)}
Let $I_2 = [s, \rho, s\omega]$ and $I_1 = [s, \rho, \omega ]$ be two ideals such that $I_1 \subseteq I_2$.  Then $J = I_2I_1^{-1} = [s, \rho, \omega]$.
\end{lemma}
\begin{proof}
This follows immediately from Proposition \ref{sing_wildram}.
\end{proof}

\begin{lemma}\label{l:splitting_s_4}{\bf (Splitting for Type IV primes)}
Let $I_2 = [s's'', s'\rho, s''(v_2 + w_2\rho + \omega)]$ and $I_1 = [s's'', \rho, v_1 + \omega ]$ be two ideals such that $I_2 \subseteq I_1$.  Then $J = I_2I_1^{-1} = [s's'', \rho, V + \omega]$, where
\begin{equation*} V \equiv E \pmod{d}, \quad V \equiv 0 \pmod{s's''/d}, \ \mbox{ and } \ d=\gcd(s'', v_1). \end{equation*}  
\end{lemma}
\begin{proof}
By Proposition \ref{splitram}, $J = [s's'', \rho, V + \omega]$ for some $V$.  For a given prime $P$, $I_2$ contains either $\mathfrak{pq}$ or $\mathfrak{q}^2$ and no higher powers, and the ideal $I_1$ contains either $\mathfrak{p}$ or $\mathfrak{q}$.  The quantity $d$ corresponds to the ramified primes in $I_1$.  For these primes the unramified conjugate is the inverse, and hence justifies the choice for $V$ modulo $d$.  
\end{proof}

Rather than proceed straight to the division propositions, we illustrate the method behind the division in Figure 1.  The hardest part of division is tracking the various products lying over completely split primes.  The figure illustrates the order of operations (as described in the proof) used to complete ideal division.  For $\mathfrak{p}$ and $\mathfrak{q}$ lying over a completely split prime $P$ we will walk through the division process in the case that the dividend is $\mathfrak{p}^8\mathfrak{q}^6$ and the divisor is $\mathfrak{p}^5\mathfrak{q}$.

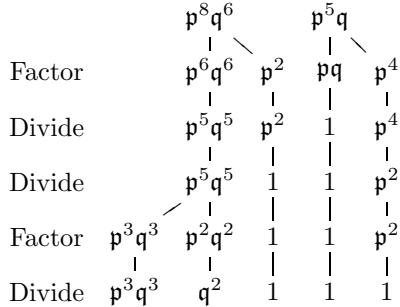
\begin{figure}[tbh]\label{div_walk}
\centerline{
\xymatrix@C=.3pc@R=.4pc{ 
 & & \mathfrak{p}^8\mathfrak{q}^6 \ar@{-}[d] \ar@{-}[dr] & & \mathfrak{p}^5\mathfrak{q}  \ar@{-}[d]  \ar@{-}[dr] & \\
\mbox{Factor} & & \mathfrak{p}^6\mathfrak{q}^6  \ar@{-}[d]&  \mathfrak{p}^2   \ar@{-}[d]    & \mathfrak{pq}    \ar@{-}[d]   &\mathfrak{p}^4 \ar@{-}[d] \\    
\mbox{Divide} & & \mathfrak{p}^5\mathfrak{q}^5  \ar@{-}[d] &  \mathfrak{p}^2     \ar@{-}[d]    &   1        \ar@{-}[d]        & \mathfrak{p}^4 \ar@{-}[d]\\
\mbox{Divide} & & \mathfrak{p}^5\mathfrak{q}^5  \ar@{-}[d] \ar@{-}[dl] &      1     \ar@{-}[d]       &          1     \ar@{-}[d]   & \mathfrak{p}^2 \ar@{-}[d]\\
\mbox{Factor} & \mathfrak{p}^3\mathfrak{q}^3 \ar@{-}[d] & \mathfrak{p}^2\mathfrak{q}^2 \ar@{-}[d] & 1 \ar@{-}[d] & 1 \ar@{-}[d] & \mathfrak{p}^2 \ar@{-}[d]\\
\mbox{Divide} & \mathfrak{p}^3\mathfrak{q}^3   & \mathfrak{q}^2 & 1 &1&  1 \\
}
} 
\caption{Division of $\mathfrak{p}^8\mathfrak{q}^6$ by $\mathfrak{p}^5\mathfrak{q}$}
\end{figure}

The tree for the dividend ends with three branches.  It should be noted that the last two nodes on the tree are relatively prime; more specifically, at least one of them is one.  This will be key for the next proof because it relies on the product of the those two nodes being relatively prime.

\begin{prop}\label{t:division_s_1}{\bf (Division for Type I and II primes)}
Let $I_i = [s_i, s_i'(u_i + \rho), v_i + w_i\rho + \omega]$  for $i = 1,2$ be such that $I_2 \subseteq I_1$.  Then $J = I_2I_1^{-1} = [S, S'(U+ \rho), V + W \rho + \omega]$, where
\begin{equation*} S = \frac{s_2}{s_1'd}, \quad S' = \frac{s_2'd}{s_1},  \quad
    U \equiv \left\{ \begin{array}{ll}
                     Iw_2 - u_1  & \pmod{ s_1/(s_1'd) }\\
                     u_2 & \pmod{  s_2/(s_2'd) }
                     \end{array}
             \right. ,
\end{equation*}
\begin{equation*}  \quad V \equiv (W - w_2)U + v_2 \pmod{S},\quad W \equiv w_2 \pmod{S'},\end{equation*}
\begin{equation*}\mbox{ and}   \ d = \gcd\left(\frac{s_2}{s_2'}, \frac{s_1}{s_1'},  u_1 - u_2  \right). \end{equation*}
\end{prop}

\begin{proof}
We begin by factoring both $I_1$ and $I_2$ into two different ideals:
\begin{equation*} I_i = I_{i,1}I_{i,2} = [ s_i', s_i'\rho, v_i + w_i\rho + \omega] \left[ \frac{s_i}{s_i'}, u_i + \rho, v_i  - u_iw_i + \omega \right]. \end{equation*}
The first division is
\begin{equation}\label{s_1_evendeg}  I_{2,1}I_{1,1}^{-1}=   \left[ \frac{s_2'}{s_1'}, \frac{s_2'}{s_1'}\rho, v_2 + w_2\rho + \omega \right]. \end{equation}
All that remains of the divisor is $I_{1,2}= \left[s_1/s_1', u_1 + \rho, v_1  - u_1w_1 + \omega \right].$
We consider the greatest common divisor of this ideal with the corresponding ideal arising from $I_2$.  This is the justification for $d$ in the proposition statement.  We perform the following division:
\begin{equation*} \left[\frac{s_2}{s_2'}, u_2 + \rho, v_2  - u_2w_2 + \omega \right] \left[ d , u_1 + \rho, v_1  - u_1w_1 + \omega \right]^{-1} =  \end{equation*}
\begin{equation*} \left[\frac{s_2}{s_2'd}, u_2 + \rho, v_2  - u_2w_2 + \omega \right], \end{equation*}
which justifies one of the two congruences for $U$.  We factor out of the ideal in \eqref{s_1_evendeg} the part that matches the remaining divisor.  That is,
\begin{equation*} \left[ \frac{s_2'}{s_1'}, \frac{s_2'}{s_1'}\rho, v_2 + w_2\rho + \omega \right] = \end{equation*}
\begin{equation}\label{s_1_split} \left[ \frac{s_2'd}{s_1}, \frac{s_2'd}{s_1}\rho, v_2 + w_2\rho + \omega \right] \left[ \frac{s_1}{s_1'd}, \frac{s_1}{s_1'd}\rho, v_2 + w_2\rho + \omega \right]. \end{equation}
We apply Lemma \ref{l:splitting_s_1} to the right hand ideal of \eqref{s_1_split} and the remainder of the divisor to get
\begin{equation*} \left[ \frac{s_1}{s_1'd}, \frac{s_1}{s_1'd}\rho, v_2 + w_2\rho + \omega \right] \left[ \frac{s_1}{s_1'd}, u_1 + \rho, v_1 -w_1u_1 + \omega \right]^{-1} \end{equation*}
\begin{equation*} =  \left[ \frac{s_1}{s_1'd}, Iw_2 - u_1 + \rho, v_2 - Iw_2^2 + u_1w_2 + \omega  \right]. \end{equation*}
This ideal gives the other congruence for $U$ and the division is complete at this step.  The choice for $S$ is justified by looking at the first term in the three ideals that remain; likewise $S'$ is the product of the coefficients of $\rho$:
\begin{equation*} S = \left(\frac{s_1}{s_1'd} \right)\left(\frac{s_2'd}{s_1} \right)\left(\frac{s_2}{s_2'd}  \right)= \frac{s_2}{s_1'd} \quad \mbox{and} \quad S' = \frac{s_2'd}{s_1}.\end{equation*}
Since $v_2 + w_2\rho + \omega \in J$, it just remains to modify this element so that it is canonical.  This justifies the choice for $V$ and $W$.  
\end{proof}

\begin{prop}\label{t:division_s_3}{\bf (Division for Type III primes)}
Let $I_i = [s_i,  \rho, s_i''\omega]$  for $i = 1,2$ be such that $I_2 \subseteq I_1$.  Then $J = I_2I_1^{-1} = [S,  \rho, S'' \omega]$, where
\begin{equation*} S = \frac{s_2}{s_1''d}, \quad S'' = \frac{s_2''d}{s_1},  \ \mbox{ and } \ d = \gcd\left( \frac{s_1}{s_1''}, \frac{s_2}{s_2''} \right) .\end{equation*}
\end{prop}

\begin{proof}
This follows by using the same arguments presented in the proof of Proposition \ref{t:division_s_1}.  The key distinction is how the ideals are factored:
\begin{equation*} I_i = [s_i'', \rho, s_i''\omega] \left[ \frac{s_i}{s_i''},\rho, \omega \right].\end{equation*}
The rest of the arguments are simplified given that these are products of  totally ramified primes.
\end{proof}

\begin{prop}\label{t:division_s_4}{\bf (Division for Type IV primes)}
Let $I_i = [s_i, s_i'(u_i + \rho), s_i''(v_i + w_i\rho + \omega)]$  for $i = 1,2$ be such that $I_2 \subseteq I_1$.  Then $J= I_2I_1^{-1} = [S, S'(U+ \rho), S''(V + W \rho + \omega)]$, where
\begin{equation*} S = \frac{s_2}{s_1's_1''d}, \quad S' = \gcd\left( \frac{ds_2's_2''}{s_1}, \frac{s_2'}{s_1'} \right),  \quad S'' = \gcd\left( \frac{ds_2's_2''}{s_1}, \frac{s_2''}{s_1''} \right), \end{equation*}
\begin{equation*}  U \equiv \left\{ \begin{array}{ll}
                     0  & \pmod{ s_1/(s_1's_1''d) }\\
                     u_2 & \pmod{  s_2/(s_2's_2''d) }
                     \end{array}
             \right. ,
\end{equation*}
\begin{equation*} d = \gcd\left(\frac{s_2}{s_2's_2''}, \frac{s_1}{s_1's_1''}, v_1 - w_1u_1 - v_2 + w_2u_2 \right),  \end{equation*}
\begin{equation*} S''V \equiv s_2''((W - w_2)U + v_2) \pmod{S}, \mbox{ and }  S''W \equiv s_2''w_2 \pmod{S'}.\end{equation*}
\end{prop}

\begin{proof}
We begin by factoring both $I_1$ and $I_2$ into the ideals $I_{i,1}$ and $I_{i,2}$ as above.  The first division is
\begin{equation}\label{s_4_evendeg}  I_{2,1}I_{1,1}^{-1}=   \left[ \frac{s_2's_2''}{s_1's_1''}, \frac{s_2'}{s_1'}\rho, \frac{s_2''}{s_1''}(v_2 + w_2\rho + \omega) \right]. \end{equation}

Proceeding as in the proof of Proposition \ref{t:division_s_1}, the next division yields 
\begin{equation*} \left[\frac{s_2}{s_2's_2''d}, u_2 + \rho, v_2  - u_2w_2 + \omega \right], \end{equation*}
which justifies the latter congruence for $U$.  We decompose the ideal on the right in \eqref{s_4_evendeg} to get a factor that matches the remaining divisor:  
\begin{equation}\label{s_4_split} \left[ \frac{s_2's_2''d}{s_1}, S'\rho, S''(v_2 + w_2\rho + \omega) \right] \left[ \frac{s_1}{s_1's_1''d}, s_3'\rho, s_3''(v_2 + w_2\rho + \omega) \right], \end{equation}
where \begin{equation*} s_3' = \gcd\left( \frac{s_1}{s_1's_1''d}, \frac{s_2'}{s_1'} \right) \mbox{ and } s_3'' = \gcd\left( \frac{s_1}{s_1's_1''d}, \frac{s_2''}{s_1''} \right). \end{equation*}
Note that $S'S'' = s_2's_2''d/s_1$ and $s_3's_3'' = s_1/s_1's_1''d$.  Apply Lemma \ref{l:splitting_s_4} to the right most ideal of \eqref{s_4_split} and the remainder of the divisor to get
\begin{equation*} \left[ \frac{s_1}{s_1's_1''d}, s_3'\rho, s_3''(v_2 + w_2\rho + \omega) \right] \left[ \frac{s_1}{s_1's_1''d}, u_1 + \rho, v_1 -w_1u_1 + \omega \right]^{-1} \end{equation*}
\begin{equation*} =  \left[ \frac{s_1}{s_1's_1''d}, \rho, v_3+ \omega  \right], \end{equation*}
where $v_3$ is given in Lemma \ref{t:inverse_s_3}.  This ideal gives the other congruence for $U$ and the division is complete at this step.  The choice for $S$ is justified by looking at the first term in the three ideals that remain:
\begin{equation*} S = \left(\frac{s_1}{s_1's_1''d} \right)\left(\frac{s_2's_2''d}{s_1} \right)\left(\frac{s_2}{s_2's_2''d}  \right)= \frac{s_2}{s_1's_1''d}.\end{equation*}
The choices for $S'$ and $S''$ are justified in \eqref{s_4_split}.
Since $s_2''(v_2 + w_2\rho + \omega) \in J$, it just remains to modify this element so that it is canonical and this justifies the choice for $V$ and $W$.  The argument here is the same as in Proposition \ref{t:division_s_1} except we have to account for the coefficient of $\omega$.  
\end{proof}

We close this section with a proposition on dividing a nonprimitive ideal by a primitive ideal.   Consider an ideal of the form $\langle d \rangle I_2$, where $I_2$ is primitive, and a primitive ideal $I_1$.  To compute $\langle d \rangle I_2 I_1^{-1}$, we begin by removing as much of $I_1$ from $\langle d \rangle$ as is possible.  The remaining factor of $I_1$ is then removed from $I_2$.  The primitive parts of the two divisions are $I_d$ and $I_m$, and their product is not necessarily primitive.    While this might seem problematic, the propositions on multiplication can be used calculate the product.  A proposition on multiplication will invoke this proposition, but it is invoked under the assumption that $I_1$ completely divides $\langle d \rangle$ and that there is no corresponding factor $I_2$.  

\begin{prop}{\bf (Nonprimitive division for Type I and II primes)}\label{t:nonprim_div_s_1}\\
Let $I_2 = d[s_2, s_2'(u_2 + \rho), v_2 + w_2\rho + \omega]$ and $I_1 = [ s_1, s_1'(u_1+\rho), v_1 + w_1\rho + \omega]$ be such that $\langle d \rangle I_2 \subseteq I_1$.  Then $IJ= I_2I_1^{-1} = (D_3)I_dI_m$, where
\begin{equation*} I_d = [s_2, s_2'(u_2 + \rho), v_2 + w_2\rho + \omega]\left[ \frac{s_1}{D_1D_2}, \frac{s_1'}{D_1}(u_1 + \rho), v_1 + w_1\rho + \omega \right]^{-1}\end{equation*}
is calculated by Proposition \ref{t:division_s_1}, 
\begin{equation*}I_m = \overline{[D_1D_2, D_1(u_1 + \rho), v_1 + w_1\rho + \omega]} \end{equation*}
is calculated by Proposition \ref{t:inverse_s_1}, 
\begin{equation*} D_1 = \gcd(s_1',d), \quad D_2 = \gcd\left(\frac{s_1}{s_1'}, \frac{d}{D_1}\right), \mbox{ and } D_3 = \frac{d}{D_1D_2}. \end{equation*}
\end{prop} 

\begin{proof}
We note that $\overline{I_m} \subseteq \langle d \rangle$ and $\overline{I_m} [ s_1/D_1D_2, s_1'/D_1(u_1 + \rho), v_1 + w_1\rho + \omega ]  = I_1$.  Therefore $\langle d \rangle \overline{I_m}^{-1} = I_m$.  After this division, the factors that remain in $I_1$ are $[ s_1/D_1D_2, s_1'/D_1(u_1 + \rho), v_1 + w_1\rho + \omega ]$ and this is contained in $I_2$.  
\end{proof}

The next two propositions are stated without proof.  The proofs follow a similar argument as the proof above and rely, like this proof, nearly entirely on the previously proved propositions.  

\begin{prop}{\bf (Nonprimitive division for Type III primes)}\label{t:nonprim_div_s_3}\\
Let $I_2 =d [s_2, \rho, s_2'' \omega]$ and $I_1 = [ s_1, \rho, s_1'' \omega]$ be such that $\langle d \rangle I_2 \subseteq I_1$.  Then $J = I_2I_1^{-1} = (D_3)I_dI_m$, where
\begin{equation*} I_d = [s_2, \rho, s_2''\omega]\left[ \frac{s_1}{D_1D_2}, \rho, \frac{s_2''}{D_1} \omega \right]^{-1}\end{equation*}
is calculated by Proposition \ref{t:division_s_3}, 
\begin{equation*}I_m = \overline{[D_1D_2, \rho, D_1 \omega]} \end{equation*}
is calculated by Proposition \ref{t:inverse_s_3}, 
\begin{equation*} D_1 = \gcd(s_1'',d), \quad D_2  = \gcd\left(\frac{s_1}{s_1''}, \frac{d}{D_1}\right), \mbox{ and } D_3 = \frac{d}{D_1D_2}. \end{equation*}
\end{prop} 

\begin{prop}{\bf (Nonprimitive division for Type IV primes)}\label{nonprim_div_s_4}\\
Let $I_2 = d[s_2, s_2'(u_2 + \rho), s_2''(v_2 + w_2\rho + \omega)]$ and $I_1 = [ s_1, s_1'(u_1+\rho), s_1''(v_1 + w_1\rho + \omega)]$ be such that $\langle d \rangle I_2 \subseteq I_1$.  Then $J = I_2I_1^{-1} = (D_4)I_dI_m$, where
\begin{equation*} I_d = [s_2, s_2'(u_2 + \rho), v_2 + w_2\rho + \omega]\left[ \frac{s_1}{D_1D_2D_3}, \frac{s_1'}{D_1}(u_1 + \rho), D_2(v_1 + w_1\rho + \omega) \right]^{-1}\end{equation*}
is calculated by Proposition \ref{t:division_s_4}, 
\begin{equation*}I_m = \overline{[D_1D_2D_3, D_1(u_1 + \rho), D_2(v_1 + w_1\rho + \omega)]} \end{equation*}
is calculated by Proposition \ref{t:inverse_s_4}, 
\begin{equation*} D_1 = \gcd(s_1',d), \ \  D_2 = \gcd(s_1'', d), \ \  D_3 = \gcd\left(\frac{s_1}{s_1's_1''}, \frac{d}{D_1D_2}\right), \mbox{ and } D_4 = \frac{d}{D_1D_2D_3}. \end{equation*}
\end{prop}


\section{Ideal Multiplication}

Theoretically, ideal multiplication is the easiest operation that will be discussed since it may be achieved by simply performing linear algebra.  The goal of these propositions is to eliminate much of the excess work that would be required to reduce the nine cross products arising in the multiplication of two ideals down to a basis.  The extreme amount  of redundancy is obvious for certain products.  For example, the product of two relatively prime ideals may be computed quickly using the Chinese Remainder Theorem.   Computationally, relatively prime operands are to  be expected and the product may be calculated as Scheidler did in Theorem 4.4 of \cite{idealarithmetic}.  

\begin{theorem}{\bf (Theorem 4.4 of \cite{idealarithmetic})}\label{t:crtforideal}
Let $I_i = [ s_i, s_i'(u_i + \rho), s_i''(v_i + w_i\rho + \omega)]$ with $i = 1,2, 3$ be two ideals such that $\gcd(s_1, s_2) = 1$.  Then $I_3 = I_1I_2$ is given by
\begin{equation*} s_3 = s_1s_2, \quad s_3' = s_1's_2', \quad s_3'' = s_1''s_2'', \end{equation*}
\begin{equation*}  u_3 \equiv \left\{ \begin{array}{ll}
                     u_1  & \pmod{ s_1/s_1' }\\
                     u_2 & \pmod{  s_1/s_2'}
                     \end{array}
             \right. ,    \quad 
              w_3 \equiv \left\{ \begin{array}{ll}
                     w_1  & \pmod{ s_1'}\\
                     w_2 & \pmod{  s_2' }
                     \end{array}
             \right. , \quad \mbox{ and }
\end{equation*}
\begin{equation*}  v_3 \equiv \left\{ \begin{array}{ll}
                     v_1 + u_1(w_3-w_1)  & \pmod{ s_1/s_1'' }\\
                     v_2 + u_2(w_3 - w_2) & \pmod{  s_2/s_2'' }
                     \end{array}
             \right. .
\end{equation*}

\end{theorem}

In contrast to cubic function fields of unit rank one, we can not assume that the two operands will be relatively prime.  Thus, we will be forced to develop ideal multiplication systematically.  The first set of propositions assumes that the product is primitive and this will be used to aid in the case where the product is not assumed to be primitive.  

In proof of Proposition \ref{t:division_s_1} the two congruences for  $U$ were sufficient in guaranteeing that $U$ was determined uniquely modulo $S/S'$.  This was because $s_1/(s_1'd)$ and $s_2/(s_2'd)$ were relatively prime.  Had they shared a common factor, $U$ would have been determined only up to the least common multiple of $s_1/(s_1'd)$ and $s_2/(s_2'd)$.  This proposition is built to handle just such a situation.

\begin{prop}{\bf (Primitive Multiplication for Type I primes)}\label{t:prim_mult_s_1}\\
Let $I_i = [s_i, s_i'(u_1 + \rho), v_i + w_i \rho + \omega]$ for $i = 1,2$ be such that $I_1I_2 = I_3$ is a primitive ideal.  Then $I_3 = [S, S'(U + \rho), V + W\rho + \omega]$, where
\begin{equation*} S = \frac{s_1s_2d_1}{d}, \quad S' = \frac{s_1's_2'd}{d_1},  \quad W = w_3 - cS', \quad V \equiv v_3 - qS'U \pmod{S},\end{equation*}
\begin{equation*} \mbox{ and } U \equiv u_3 + k\frac{s_1s_2d_1}{s_1's_2'd^2} \pmod{S/S'}. \end{equation*}
We choose $c$  to make $\deg W$ minimal and define $u_3$, $v_3$, $w_3$, $d$, $d_1$ as follows
\begin{align*} d &= \gcd\left( \frac{s_1}{s_1'}, \frac{s_2}{s_2'} \right), & d_1&= \gcd(d, u_1 - u_2), \\ 
 u_3 &\equiv u_1 \pmod{s_1d_1/s_1'sd}, & u_3 &\equiv u_2 \pmod{s_2d_1/s_2'd}, \end{align*}
\begin{equation*} \mbox{ and k is chosen such that \quad } d_1 \ \vrule \ \frac{(u_3^3 - u_3A - FI^2)S'd_1}{S} + kA, \end{equation*} 
\begin{align*}
w_3 &=  a_1s_2w_1 + a_2s_1w_2 + a_3s_1's_2'(u_1 + u_2) + a_4s_1'(v_2+u_1w_2) \\
& \quad + a_5s_2'(v_1 + u_2w_1) + a_6(v_1w_2 + v_2w_1 - F) \\
v_3&  =  a_1s_2v_1 + a_2s_1'v_2 + a_3s_1's_2'(u_1u_2 + A) + a_4s_1's(u_1v_2 - FI + w_2) \\
&\quad + a_5s_2'(u_2v_1 - FI + w_1A) + a_6(v_1v_2 + w_1w_2-w_1FI - W_2FI) 
\end{align*}
and $a_1$, $a_2$,$a_3$, $a_4$, $a_5$, and $a_6$  are given by the extended euclidian algorithm as:
\begin{align*}
1 & =  a_1s_2 + a_2s_1 + a_3s_1's_2'I + a_4s_1'(u_1 + Iw_2) \\
& \quad+ a_5s_2'(u_2 + Iw_1) + a_6(v_1 + v_2 + w_1w_2I - E). 
\end{align*}
\end{prop}           

\begin{proof}
Since we assume $I_3$ is primitive, it has a canonical basis of the form claimed.  We begin by factoring $I_1$ and $I_2$ and deal with their product using smaller and simpler ideals.  The easiest part of the product is 
\begin{equation*} [s_1', s_1'\rho, v_1 + w_1\rho + \omega ] [s_2', s_2'\rho, v_2 + w_2\rho + \omega ] = [s_1's_2', s_1's_2'\rho, V + W\rho + \omega ]. \end{equation*} 
While we still need to find congruences for $V$ and $W$, we will return to those later and focus on the difficult part of the product:
\begin{equation}\label{singles_s_1} \left[ \frac{s_1}{s_1'}, u_1 + \rho, v_1 - w_1u_1 + \omega \right]  \left[ \frac{s_2}{s_2'}, u_2 + \rho, v_2 - w_2u_2 + \omega \right]. \end{equation}
The goal will be two split this product up into two factors.  The quantity $d$ signifies common possible prime factors in this product, and $d_1$ indicates those primes that appear as squares in the product.  Thus, we write the above product as
\begin{equation*} \left[ \frac{S}{S'}, U + \rho, V + \omega \right] \left[\frac{d}{d_1}, \frac{d}{d_1}\rho, V + W\rho + \omega \right]. \end{equation*}
We conclude from this that $S' = s_1's_2'd/d_1$ and by equating norms that $S = s_1s_2d/d_1$.  Combining the two previous statements we see that
\begin{equation*} \left[ \frac{S}{S'}, U + \rho, V + \omega \right]  =  \left[ \frac{s_1d_1}{s_1'd}, u_1 + \rho, v_1 - w_1u_1 + \omega \right]  \left[ \frac{s_2d_1}{s_2'd}, u_2 + \rho, v_2 - w_2u_2 + \omega \right]. \end{equation*}
This justifies the choice for $u_3$, and note that $u_3$ is defined uniquely modulo the least common multiple of $s_1d_1/s_1'd$ and $s_2d_1/s_2'd$.  Thus we can write $U = u_3 + kS/S'd_1$ and consider 
\begin{equation*}\frac{S}{S'} \ \vrule \ N(U + \rho) \Rightarrow \frac{S}{S'} \ \vrule \ (u_3^3 - u_3A - FI^2) + kA\frac{S}{S'd_1}. \end{equation*}
From the definition of $u_3$, $S/S'd_1$ divides  $u_3^3 - u_3A - FI^2$ so we can conclude
\begin{equation*} d_1 \ \vrule \ \frac{ (u_3^3 - u_3A - FI^2)S'd_1}{S} + kA \end{equation*}
as claimed.  This determines $U$ modulo $S/S'$ as needed.  To calculate $V$ and $W$ we find any element of the form $v_3 + w_3\rho + \omega \in I_3$.  Since $I_3$ is primitive and contains no index divisors, the greatest common divisor of the coefficients of $\omega$ arising from all possible products of basis elements of $I_1$ and $I_2$ must be $1$.   Once this element is computed, it is a matter of subtracting multiples of the two previously calculated basis elements to ensure  the third element is canonical. 
\end{proof}

The calculation of $W$ and $V$ is not as difficult as it looks.  As we noted before, if $s_1$ and $s_2$ are relatively prime the above proposition is superfluous and the multiplication can be done via the Chinese Remainder Theorem.  Assuming $s_1$ and $s_2$ are not relatively prime, we still expect that we will be able to write 1 as a linear combination of fewer than all six terms.  

\begin{prop}{\bf (Primitive Multiplication for Type II primes)}\label{t:prim_mult_s_2}
Let $I_i = [s_i, s_i'(u_1 + \rho), v_i + w_i \rho + \omega]$ for $i = 1,2$ be such that $I_1I_2 = I_3$ is a primitive ideal.  Then $I_3 = [S, S'(U + \rho), V + W\rho + \omega]$, where
\begin{equation*} S =  s_1s_2/d,  \quad S' = ds_1's_2',\quad d = \gcd\left( \frac{s_1}{s_1'}, \frac{s_2}{s_2'} \right), \end{equation*}
\begin{equation*} U \equiv f \pmod{S/S'}, \quad W \equiv I^{-1}f \pmod{S'}, \quad V \equiv f^2I^{-1} \pmod{S}, \end{equation*}
where $f$ is defined by  $f^3 \equiv FI^2 \pmod{S}$.
\end{prop}           

\begin{proof}
We invoke Proposition \ref{t:wildrambasis} to calculate $U$, $V$, and $W$.
\end{proof}

\begin{prop}{\bf (Primitive Multiplication for Type III primes)}\label{t:prim_mult_s_3}
Let $I_i = [s_i, \rho, s_i''\omega]$ for $i = 1,2$ be such that $I_1I_2 = I_3$ is a primitive ideal.  Then 
\begin{equation*}I_3 = \left[ \frac{s_1s_2}{d}, \rho,  (s_1''s_2''d)\omega \right] \mbox{ where } d = \gcd\left(\frac{s_1}{s_1''}, \frac{s_2}{s_2''} \right). \end{equation*}
\end{prop}           

\begin{proof}
This follows from Proposition \ref{sing_wildram}.
\end{proof}

\begin{prop}{\bf (Primitive Multiplication for Type IV primes)}\label{t:prim_mult_s_4}
Let $I_i = [s_i, s_i'(u_1 + \rho), s_i''(v_i + w_i \rho + \omega)]$ for $i = 1,2$ be such that $I_1I_2 = I_3$ is a primitive ideal.  Then $I_3 = [S, S'(U + \rho), S''(V + W\rho + \omega)]$, where
\begin{equation*} S = \frac{s_1s_2}{dd_1d_2}, \quad S' = s_1's_2'd, \quad S'' = s_1''s_2''d_1d_2. \end{equation*}
To define $d_1$, $d_2$, and $d$, let
\begin{equation*} s_{\mathfrak{q}i} = \gcd\left( \frac{s_i}{s_i's_i''}, v_i-w_iu_i \right) \mbox{ for } i = 1,2, \end{equation*}
then
\begin{equation*} d = \gcd( s_{\mathfrak{q}1}, s_{\mathfrak{q}2}), \quad d_1 = \gcd\left( s_{\mathfrak{q}2}, \frac{s_1}{s_1's_1''s_{\mathfrak{q}1}} \right), \mbox{ and } \  d_2 = \gcd\left( s_{\mathfrak{q}1}, \frac{s_2}{s_2's_2''s_{\mathfrak{q}2}} \right). \end{equation*}
We defined $U = (s_1''s_2''d_1d_2)u_3 $ where $u_3$ satisfies 
\begin{equation*}u_3 \equiv u_1 \pmod{\frac{s_1}{s_1''s_1''dd_1d_2}}, \ \mbox{ and } \ u_3 \equiv u_2 \pmod{ \frac{s_2}{s_2's_2''dd_1d_2} }. \end{equation*}  
Finally we can choose $V$ and $W$ as
\begin{equation*} S''W = S''w_3 - qS', \quad S''V \equiv S''v_3 - qS'U \pmod{S},\end{equation*}
where $q$ is chosen so that $\deg V$ and $\deg W$ are minimal and
\begin{align*}
S''w_3 = & a_1s_2s_1''w_1 + a_2s_1s_2''w_2 + a_3s_1's_2'(u_1 + u_2) + a_4s_1's_2''(v_2+u_1w_2) \\
&+ a_5s_2's_1''(v_1 + u_2w_1) + a_6s_1''s_2''(v_1w_2 + v_2w_1 - F) 
\end{align*}
\begin{align*}
S''v_3 = & a_1s_2s_1''v_1 + a_2s_1s_2''v_2 + a_3s_1's_2'(u_1u_2 + A) + a_4s_1's_2''(u_1v_2 - FI + w_2) \\
&+ a_5s_2's_1''(u_2v_1 - FI + w_1A) + a_6s_1''s_2''(v_1v_2 + w_1w_2-w_1FI - W_2FI) 
\end{align*}
where $a_i$ for $i =1, \ldots, 6$ come from the extended greatest common divisor calculation,
\begin{align*}
S''= & a_1s_2s_1'' + a_2s_1s_2'' + a_3s_1's_2'I + a_4s_1's_2''(u_1 + Iw_2) \\
&+ a_5s_2's_1''(u_2 + Iw_1) + a_6s_1''s_2''(v_1 + v_2 + w_1w_2I - E). 
\end{align*}
\end{prop}           

\begin{proof}
The details of the proof are similar to those above.  We will try and note only the key distinctions.  We factor $I_i$ into three factors as
\begin{equation*} I_i = J_{i,1}J_{i,2}J_{i,3} = \left[\frac{s_i}{s_i's_i''}, u_i + \rho, v_i-w_iu_i + \rho \right][s_i'', \rho, s_i''\omega][s_i', s_i'\rho, v_i + w_i\rho + \omega]. \end{equation*}
Since $I_3$ is primitive, all three of $\gcd(s_2'',s_1''),  \gcd(s_2',s_1''),$ and $ \gcd(s_2'',s_1')$ are one.
This simplifies the number of possible products to consider.  We factor $J_{i,1}$ further to distinguish ramified primes (denoted with a subscript $\mathfrak{q}$) from the unramified primes:
\begin{equation*}J_{i,1} = [ s_{\mathfrak{q}i}, \rho, \omega] \left[\frac{s_i}{s_i's_i''s_{\mathfrak{q}i} }, u_i  + \rho, v_i - w_iu_i + \omega \right].\end{equation*}
Now there are three possible type of products these two ideals can form.  Products corresponding to a common place of $\Fq(x)$ lying below $\mathfrak{p}$ and $\mathfrak{q}$ indicate the presence of that polynomial being a factor of the coefficent of $\omega$.  This justifies the choice of $d_1$ and $d_2$.  There are at most single powers of $\mathfrak{q}$ in either of the two ideals that correspond to that part of the factorization.  Their greatest common divisor justifies the choice of $d$.  We remove these factors from their corresponding ideals in $J_{i,1}$.  We can choose $u_3$ from these two divisors of $J_{i,1}$.  This gives $u_3$ unique modulo
$S/(S'S'')$.  Since $S''$ divides $U$ this justifies the choice of $U$.  Lastly, we chose $V$ and $W$ in the same manner as in the previous proposition.  However, the fact that these ideals correspond to index divisors means that the greatest common divisor of the terms with $\omega$ will no longer be 1 but $S''$. 
 \end{proof}

Much like Proposition \ref{t:prim_mult_s_1} the greatest common divisor calculation looks complicated but in general $S''$ can be found with fewer terms than the 6 given.

Now we deal with the case that the product of two ideals is not primitive.  The key to these propositions is finding and removing the nonprimitive factors.  The remaining product is primitive and the previous propositions may be invoked.  

\begin{prop}{\bf (Multiplication for Type I primes)} \label{t:mult_s_1}
 For $i = 1,2$ let $I_i = [s_i, s_i'(u_1 + \rho), v_i + w_i \rho + \omega]$  be two ideals .  Then $I_1I_2 = (D)I_3$ where $I_3 = I_1'I_2'J$ and $D = D_1D_2D_3$ and these quantities are as follows:
 \begin{equation*} D_1 = \gcd(s_2', s_1/s_1', u_1 + Iw_2), \quad  D_2 = \gcd(s_1', s_2/s_2', u_2 + Iw_1),\end{equation*}
\begin{equation*} D_3 = \frac{\gcd(s_1'/D_2,s_2'/D_1)}{\gcd(s_1'/D_2,s_2'/D_1, w_1-w_2)}, \end{equation*}
\begin{equation*} I_1' = \left[ \frac{s_1}{D_1D_2D_3}, \frac{s_1'}{D_2D_3}(u_1 + \rho), v_1 + w_1\rho + \omega \right], \end{equation*}
\begin{equation*} I_2' = \left[ \frac{s_2}{D_1D_2D_3}, \frac{s_2'}{D_1D_3}(u_2 + \rho), v_2 + w_2\rho + \omega \right], \mbox{ and } \end{equation*}
\begin{equation*} J = \langle D_3 \rangle \left( \overline{ [D_3, D_3\rho, v_1 + w_1\rho + \omega] } \ \overline{[D_3, D_3\rho, v_2 + w_2\rho + \omega] } \right)^{-1}, \end{equation*}
and the last calculation is done by invoking Propositions \ref{t:inverse_s_1},  \ref{t:nonprim_div_s_1}, and \ref{t:prim_mult_s_1}.
\end{prop}

\begin{proof}
We factor $I_1$ and $I_2$ as in Proposition \ref{t:division_s_1},

\begin{equation*} I_i = I_{i,1}I_{i,2} = [s_i', s_i'\rho, v_i + w_i\rho + \omega] \left[ \frac{s_i}{s_i'}, u_i + \rho, v_i + w_i\rho + \omega \right]\end{equation*}

Of these four factors the non-primitive part of the product does not arise from $I_{1,2}I_{2,2}$.  We find the non-primitive part from the product $I_{1,2}I_{2,1}$ (resp. $I_{2,2}I_{1,1}$).  It suffices to consider the coefficient of $\omega$.  Hence $D_1 = \gcd(s_2', s_1/s_1', u_1 + Iw_2)$ (resp.  $D_2 = \gcd(s_1', s_2/s_2', u_2 + Iw_1)$).  We remove $D_1$ (resp. $D_2$) from $I_{1,1}$ and $I_{2,2}$  (resp. $I_{2,1}$ and $I_{1,2}$) and rename as follows:
\begin{equation*} I_{1,2}' = \left[\frac{s_1}{s_1'D_1}, u_1 + \rho, v_1 + w_1\rho + \omega \right] , \quad I_{1,1}'= \left[ \frac{s_1'}{D_2},\frac{s_1'}{D_2}\rho, v_1 + w_1\rho + \omega \right]  \end{equation*}
\begin{equation*} I_{2,2}' = \left[ \frac{s_2}{s_2'D_2}, u_2 + \rho, v_2 + w_2\rho + \omega \right] , \ \mbox{ and } \ I_{2,1}' = \left[ \frac{s_2'}{D_1},\frac{s_2'}{D_1}\rho, v_2 + w_2\rho + \omega \right]. \end{equation*}
The  product $I_1I_2$ now has the form $(D_1D_2)I_{1,1}' I_{1,2}' I_{2,1}' I_{2,2,}'$ and any remaining nonprimitive factor comes from $I'_{1,1}I'_{2,1}$.  

Let 
\begin{equation*} I_{1,3} = [D_3, D_3\rho, v_1 + w_1\rho + \omega ] \mbox{ and } I_{2,3} = [D_3, D_3\rho, v_2 + w_2\rho + \omega ], \end{equation*}
where $D_3$ is defined above.  The choice of $D_3$ is justified because $\gcd(s_1'/D_2, s_2'/D_1)$ is the possible primes that could be part of the non-primitive product.  However, the previous greatest common divisor contains too many primes.  For a given prime $P$ we need to be able to distinguish between $\mathfrak{pq}$ and $\mathfrak{p}^2$.  If $w_1-w_2 = 0$ then the associated primes correspond to a square and that justifies the choice for the denominator in $D_3$.  We justify the claim for the ideal $J$ by noting the following equalities.  
\begin{align*}
I_{1,3}I_{2,3} &= I_{1,3}I_{2,3}\overline{I_{1,3}} \ \overline{I_{2,3}} (\overline{I_{1,3}} \ \overline{I_{2,3}})^{-1} \\
              &= \langle D_3 \rangle ^2 (\overline{I_{1,3}}\  \overline{I_{2,3}})^{-1}\\
              &= \langle D_3 \rangle \left( \langle D_3 \rangle /  (\overline{I_{1,3}} \ \overline{I_{2,3}})\right)\\
              & = \langle D_3 \rangle J
\end{align*}
The last ideal is the one given in the proposition statement and it is primitive.  We remove the factor $I_{1,3}$ from $I_{1,2}'$ and $I_{2,3}$ from $I_{2,2}'$ to get the other two primitive ideals.  The product of these three ideals is primitive and can be calculated by Proposition \ref{t:prim_mult_s_1}. 
\end{proof}

The totally ramified primes will be much easier to deal with.  For the two types, appealing to the propositions that govern their powers from Section 7 will be sufficient.  

\begin{prop}{\bf (Multiplication for Type II primes)}\label{t:mult_s_2}
 For $i = 1,2$ let $I_i = [s_i, s_i'(u_1 + \rho), v_i + w_i \rho + \omega]$  be two ideals such that $I_1I_2 = (D)I_3$ with $I_3 = I_1'I_2'J$ and $D = D_1D_2D_3$.  These quantities are given as follows:
  \begin{equation*} D_1 = \gcd \left( \frac{s_1}{s_1'}, s_2' \right), \quad D_2 = \gcd \left( \frac{s_2}{s_2'}, s_1' \right), \quad D_3 =  \gcd(s_2', s_1'),\end{equation*}
\begin{equation*} I_1' = \left[ \frac{s_1}{D_1D_2D_3}, \frac{s_1'}{D_2D_3}(u_1 + \rho), v_1 - w_1u_1 + \omega \right], \end{equation*}
\begin{equation*} I_2' = \left[ \frac{s_2}{D_1D_2D_3}, \frac{s_2'}{D_1D_3} (u_2 + \rho), v_2 - w_2u_2 + \omega \right], \end{equation*}
\begin{equation*} J = [ D_3,  f + \rho, I^{-1}f^2 + \omega ], \end{equation*}
with $f$ satisfying $f^3 \equiv FI^2 \pmod{D_3}$.
\end{prop}

\begin{proof}
Much like the previous proposition, the key is to factor the ideals and find where the nonprimitive factors arise.  Unlike the previous proposition, constructing the equivalent ideal $J$ is trivial.  This is because $D_3$ is squarefree and Proposition \ref{t:wildrambasis} states the form of these ramified primes.  
\end{proof}

\begin{prop}{\bf (Multiplication for Type III primes)}\label{t:mult_s_3}
 For $i = 1,2$ let $I_i = [s_i, \rho, s_i''\omega]$  be two ideals such then $I_1I_2 = (D)I_3$.  $I_3 = I_1'I_2'J$ and $D = D_1D_2D_3$ where these quantities are given as follows:
  \begin{equation*} D_1 =  \gcd\left( \frac{s_1}{s_1''}, s_2'' \right), \quad D_2 = \gcd \left( \frac{s_2}{s_2''}, s_1'' \right), \quad D_3 =  \gcd(s_2'', s_1''),\end{equation*}
\begin{equation*} I_1' = \left[ \frac{s_1}{D_1D_2D_3}, \rho , \frac{s_1''}{D_2D_3}\omega \right], \quad I_2' = \left[ \frac{s_2}{D_1D_2D_3}, \rho, \frac{s_2''}{D_1D_3}  \omega \right], and \end{equation*}
\begin{equation*} J = [ D_3,  \rho,  \omega ]. \end{equation*}
\end{prop}

\begin{proof}
This follows in the same manner as the previous proof.
\end{proof}

\begin{prop}{\bf (Multiplication for Type IV primes)}\label{t:mult_s_4}
 For $i = 1,2$ let $I_i = [s_i, s_i'(u_1 + \rho), s_i''(v_i + w_i \rho + \omega)]$  be two ideals.  Then $I_1I_2 = (D)I_3$ where 
 \begin{equation*}I_3 = I_1'I_2'J \mbox{ and } D = D_1D_2D_3D_4D_5D_6D_7.\end{equation*}
 These quantities are defined as follows:
 \begin{equation*} D_1 =  \gcd(s_2'',s_1/s_1's_1'', v_1 - u_1w_1), \quad  D_2 = \gcd(s_1'', s_2/s_2's_2'', v_2-w_2u_2),\end{equation*}
 \begin{equation*}D_3 = \frac{ \gcd(s_1/s_1's_1'', s_2'') } { \gcd(s_1/s_1's_1'', s_2'', v_1 - w_1u_1) }, \quad  D_3 = \frac{ \gcd(s_2/s_2's_2'', s_1'') } { \gcd(s_2/s_2's_2'', s_1'', v_2 - w_2u_2) },\end{equation*}
\begin{equation*} D_5 = \gcd( s_1'/D_4, s_2''/D_1),  \quad D_6 = \gcd(s_2'/D_3, s_1''/D_2)  , \end{equation*}
\begin{equation*} D_7 = \gcd\left( \frac{s_2''}{D_1D_5}, \frac{s_1''}{D_2D_6} \right)\end{equation*}
\begin{equation*} I_1' = \left[ \frac{s_1}{D_2D_4D_5D_6D_7}, \frac{s_1'}{D_4D_5}\rho, \frac{s_1''}{D_2D_6D_7}(v_1 + w_1\rho + \omega) \right], \end{equation*}
\begin{equation*} I_2' = \left[ \frac{s_2}{D_1D_3D_5D_6D_7}, \frac{s_2'}{D_3D_6}\rho,  \frac{s_2''}{D_1D_5D_7}(v_2 + w_2\rho + \omega) \right], \mbox{ and } \end{equation*}
\begin{equation*} J = [D_5D_6, \rho, \omega] [D_7, \rho, \omega + E ].\end{equation*}
\end{prop}

\begin{proof}
The proof follows in a similar manner as the previous three proofs.  Again, we seek only to highlight the differences.  We begin by factoring $I_1$ into three ideals as
\begin{equation*} I_1 = I_{1,1}I_{1,2}I_{1,3} = \left[ \frac{s_1}{s_1's_1''}, u_1 + \rho, v_1 - w_1u_1 + \omega \right][s_1', s_1'\rho, v_1 + w_1\rho + \omega][s_1'', \rho, s_1''\omega], \end{equation*}
and likewise with $I_2$.  The quantity $D_1$ (resp. $D_2$, $D_3$, $D_4$) is the nonprimitive part from $I_{1,1}I_{2,3}$ (resp. $I_{2,1}I_{1,3}$, $I_{1,1}I_{2,2}$, $I_{2,1}I_{1,2}$).  We remove these factors from the ideals and consider $I_{1,2}I_{2,3}$ (resp. $I_{2,2}I_{1,3}$).  Here we are considering the case in which one ideal contains squares of the ramified prime (say, $\mathfrak{q}^2$) and the other ideal contains products of a ramified prime with its corresponding unramified prime (say, $\mathfrak{pq}$).  The product of $\mathfrak{q^2}\mathfrak{pq}$ is $(P)\mathfrak{q}$.  Thus we get $(D_5)[D_5, \rho, \omega]$ (resp $(D_6)[D_6, \rho, \omega]$).  We remove the factor $D_5$ (resp. $D_6$) from $I_{1,2}$ and $I_{2,3}$ (resp. $I_{2,2}$ and $I_{1,3}$) and consider one last product of $I_{1,3}I_{2,3}$.  This is a product where each ideal has primes of the form $\mathfrak{pq}$ and therefore the product must be of the form $(P)\mathfrak{p}$.  We get $(D_7)[D_7, \rho, \omega + E]$.

\end{proof}

We have stated the basic ideal operations necessary for arithmetic.  The key now is to give a method to find a distinguished element in an ideal class.  From this point forward, $\mathcal{F}/K$ will be assumed to have a totally ramified infinite place with $3\nmid \deg{FI^2}$.  This latter assumption is necessary since we rely on Theorem \ref{t:n_norm}.   These assumptions also ensure that the ideal class group is isomorphic to the Jacobian of the curve.


\section{Elements of Minimal Norm}

The content in this section closely mirrors Section 8 of \cite{bauer}.  Given an ideal $J = [s, s'(u + \rho),s''( v + w\rho + \omega)]$, we want to find an element in this ideal that has minimal norm.  The existence of such an element (up to a constant scalar) is guaranteed by Theorem \ref{t:minnormexists}.  Writing the ideal as a triangular matrix and assigning a weight to each column corresponding to the norm of the associated element, the algorithm proceeds to perform elementary row operations on the matrix to find an element of minimum weight.
\\
\begin{center}
\setlength{\overfullrule}{0pt}
\vspace*{-3em}
\hspace{-1.5in}
\begin{minipage}[t]{0in}
\begin{equation*} \left[ \rule{0in}{0.3in} \right. \end{equation*}
\end{minipage}
\begin{minipage}[t]{.85in}
\begin{equation*}  \begin{array}{ccc}
s   & 0  & 0 \\
s'u & s' & 0 \\
s''v   & s''w  & s'' \\
\uparrow & \uparrow & \uparrow \\
3*\deg{FI^2} \ \ \ & 3*\deg+\deg{FI^2} \ \ \ & 3*\deg+ \deg{F^2I} 
\end{array} \end{equation*}
\end{minipage}
\begin{minipage}[t]{.5in}
\begin{equation*} \left. \hspace{2.3in} \rule{0in}{.3in} \right] \end{equation*}
\end{minipage}
\setlength{\overfullrule}{5pt}
\end{center}
\ \\
By Theorem \ref{t:n_norm}, the weights of distinct columns lie in distinct residue classes modulo three.  If two rows have their weight coming from the same position it is possible to reduce the weight of one of the rows while still maintaining a basis for the ideal.  The algorithm below just encodes the order in which to do the minimization. 

\begin{pseudocode}{MinElement} \label{a:n_min1}
\INPUT Minimal Element Algorithm.  Let $I=[s, s'(u + \rho),s''(v + w\rho + \omega)]$.
\OUTPUT $\alpha \in I$ non-zero so that $N(\alpha)$ has minimal degree.
\PRECOMP Use the ideal to define $b_1=(b_{1,1},b_{1,2},b_{1,2})=(s,0,0)$,
$b_2=(b_{2,1},b_{2,2},b_{2,2})=(s'u,s',0)$, and
$b_3=(b_{3,1},b_{3,2}, b_{3,2})=(s''v,s''w,s'')$.  Assign weights
$w_{i,1}=3 \deg b_{i,1}$, $w_{i,2}=3 \deg b_{i_2} + \deg{FI^2}$, and,
$w_{i,3}=3 \deg b_{i,3} +  \deg{F^2I}$.
\NUMLINE Set $w_i=\max \{
w_{i,1},w_{i,2}, w_{i,3} \}$, and choose $a_i$ so that
$w_i=w_{i,a_i}$ (i.e., $w_i=w_{i,a_i}=\deg N(b_i)$). Order the $b_i$
and their associated values so that $w_1\leq w_2\leq w_3$.
\WHILE{$a_1=a_2$ or $a_2=a_3$ or $a_1=a_3$}
\CASE{I}{$a_1=a_2$}
\NUMLINE $b_{2,a_2}=b_{1,a_1}c+r$
\NUMLINE replace $b_2:=b_2-cb_1$ and recalculate $a_2,w_2$.
\ENDCASE
\CASE{II}{$a_1=a_3$}
\NUMLINE $b_{3,a_3}=b_{1,a_1}c+r$
\NUMLINE replace $b_3:=b_3-cb_1$ and recalculate $a_3,w_3$.
\ENDCASE
\CASE{III}{$a_2=a_3$}
\NUMLINE $b_{3,a_2}=b_{2,a_2}c+r$
\NUMLINE replace $b_3:=b_3-cb_2$ and recalculate $a_3,w_3$.
\ENDCASE
\NUMLINE Reorder the $b_i$'s and associated values.
\ENDWHILE
\NUMLINE \textbf{Return:} $b_{1,1}+b_{1,2}\rho+b_{1,3}\omega$, the element of minimal norm.
\end{pseudocode}
Now that we can calculate an element of minimal norm, our goal will be to construct a canonical basis for the principal ideal generated by this element.


\section{Canonical Basis}

The algorithm for finding a canonical basis for a principal ideal generated by an element of $\mathcal{O}_\mathcal{F}$ is straightforward. 

\begin{pseudocode}{CanBasis}\label{a:canbasis}
\INPUT $a + b\rho + c\omega \in \mathcal{O}_F$
\OUTPUT A canonical basis of the ideal $I = \langle \alpha \rangle$.
\NUMLINE Create the matrix
\begin{equation*}
\left[ \begin{array}{lll} a & b & c \cr bA-cFI & a & bI \cr -bFI & -cF & a-cE \cr
                          \end{array} \right].
\end{equation*}
\NUMLINE Using elementary row operations transform it into a lower triangular matrix
\begin{equation*}
\left[ \begin{array}{lll} c_3 & 0 & 0 \cr c_2 & b_2 & 0 \cr c_1 & b_1
                          & a_1 \cr \end{array} \right] .
\end{equation*}
\NUMLINE Set $d = \gcd(a_1, b_2)$, $s = c_3/d$,  $s' = b_2/d$, $s'' = a_1/d$ and $u \equiv c_2/(s'd) \pmod{s/s'}$. 
\NUMLINE Compute $c$ and $w$ such that $b_1/d = s'c + w$ and $deg(w) < deg(s')$.
\NUMLINE Compute $ v\equiv c_1/d - s'qu \pmod{s}$.
\NUMLINE \textbf{Return:} The ideal $d \ [s, s'(\rho + u), s''\omega + w\rho + v]$ generated by $\alpha$, given in terms of a canonical basis.
\end{pseudocode}

Since we used only elementary row operations, the algorithm gives a valid $\Fq[x]$-basis for the principal ideal generated by $a + b\rho + c\omega$.  The latter steps in the algorithm ensure the basis is canonical.


\section{Composition and Reduction in the ideal class group}

We have all the tools we need to do composition and reduction in the ideal class group.  Given two ideals $I_1$ and$I_2$ we find a distinguished representative in the class of $I_1I_2$ as follows:  

\begin{pseudocode}{CompRed} \label{a:n_compred}
\INPUT Two ideals  $I_1$ and $I_2$ with canonical representations.
\OUTPUT The distinguished ideal $J$ equivalent to $I_1I_2$.
\NUMLINE Calculate $I_3 = I_1I_2$.
\NUMLINE  Find $\overline{I_3}$. 
\NUMLINE Find $\alpha \in \overline{I_3}$ of minimal norm using Algorithm \ref{a:n_min1}.
\NUMLINE Compute $\langle \alpha \rangle = \langle d \rangle [ s, s'(u + \rho), v + w\rho + \omega]$ using Algorithm \ref{a:canbasis}.
\NUMLINE Compute $J = \langle \alpha \rangle / \overline{I_3}$.
\NUMLINE \textbf{Return:} $J$.
\end{pseudocode}

The proof of correctness has been established in the previous sections by invoking the appropriate theorems.     For almost all cubic function field in characteristic three with a totally ramified place at infinity, we have given composition and reduction in the ideal class group.  There are, however, some exceptions - see  Example 1 in Section 6 for a function field with a totally ramified place for which the above algorithm will fail to succeed at reduction in the ideal class group.


\section{Example Computation}

We present an example to illustrate the algorithms.  The field of constants is $\mathbb{F}_{3^{10}} = \mathbb{F}_3[\alpha]/(\alpha^{10} - \alpha^6 - \alpha^5 -\alpha^4 + \alpha - 1)$ and the cubic function field is $\mathbb{F}_{3^{10}}(x,y)$ where $y$ is a root of $T^3 - T + x^4 + \alpha$.    Since $A = 1$, this is an Artin-Schreier extension with no finite ramification.   The infinite place is totally ramified and the genus of the function field is 3.    We let
\begin{equation*} I_1 = [x, -\alpha^9 -\alpha^8 + \alpha^7 + \alpha^6 +\alpha^5 - \alpha^4 - \alpha^3 + \rho, \alpha^8 - \alpha^6 -\alpha^5 +\alpha^4 -\alpha^3 + \alpha^2 + \omega] \end{equation*}
and we will find the reduced ideal in the class of $I_1^6$ following Algorithm \ref{a:n_compred}.

\textbf{Step 1.}   We calculating $I_1^2$ and $I_1^3$ followed by $I_1^6$ invoking Proposition \ref{t:prim_mult_s_1} each time.   We state only the the parameters used to define $I_1^6$ which has the form $[s_2, u_2 + \rho, v_2 + \omega]$, where
\begin{equation*} s_2 =  x^6, \quad u_2 =  x^4 - \alpha^9 - \alpha^8 + \alpha^7 + \alpha^6 + \alpha^5 -\alpha^4 - \alpha^3, \mbox{ and } \end{equation*}
\begin{equation*} v_2 =  (-\alpha^9 -\alpha^8 + \alpha^7 + \alpha^6 + \alpha^5 -\alpha^4 -\alpha^3)x^4 - \alpha^8 + \alpha^6 + \alpha^5 - \alpha^4 + \alpha^3 - \alpha^2 -1. \end{equation*}

\textbf{Step 2.}  We compute $I_3 = \overline{I_2}$.  It is clear that this inverse will have the form $[s_3, s_3\rho, v_3 + w_3\rho + \omega]$.  By appealing to Proposition \ref{t:inverse_s_1}, we have $s_3 = x^6$, $v_3 = -v_2$, and $w_3 = -u_2$.

\textbf{Step 3.}  We apply Algorithm \ref{a:n_min1} to the above ideal.  We note that the while-loop finishes in two iterations to give $a + b\rho + c\omega$ as the element of minimal norm, where
\begin{eqnarray*}
a & = & (\alpha^8 - \alpha^6 - \alpha^5 + \alpha^4 -\alpha^3 + \alpha^2 -1)x^2\\
b& = & (\alpha^9 + \alpha^8 - \alpha^7 - \alpha^6 - \alpha^5 + \alpha^4 + \alpha^3)x^2\\
c & = & x^2
\end{eqnarray*}

\textbf{Step 4.}  Applying Algorithm \ref{a:canbasis} to the above parameters gives 
\begin{equation*} \langle x^2 \rangle [ x^4, x^4\rho,  (\alpha^8 - \alpha^6 - \alpha^5 + \alpha^4 -\alpha^3 + \alpha^2 -1) + (\alpha^9 + \alpha^8 - \alpha^7 - \alpha^6 - \alpha^5 + \alpha^4 + \alpha^3)\rho + \omega ]. \end{equation*}

\textbf{Step 5.}  Finally, we calculate $\langle \alpha \rangle / I_3$ according to Proposition \ref{t:nonprim_div_s_1}. This has the form $[s_4, u_4 + \rho, v_4 + \omega]$ where 
 \begin{eqnarray*}
s_4 & = & x^2,\\
u_4& = &-\alpha^9 -\alpha^8 + \alpha^7 + \alpha^6 +\alpha^5 - \alpha^4 - \alpha^3,\\
v_4 & = & \alpha^8 - \alpha^6 -\alpha^5 +\alpha^4 -\alpha^3 + \alpha^2.
\end{eqnarray*}
Note that this happens to be $I_1^2$.

\section{Conclusion}

This work was chiefly motivated by two sources.  First we wanted to find comparable results of \cite{exp_cff,aacff} in the characteristic 3 case.  Given the history of the subject, it might have been sensible to attempt to compute fundamental units as in \cite{fun_unit}.  However, a naive approach with the integral basis given in Section 3 failed to work and we are currently investigating a solution.  However, the second motivation was a generalization along the lines of Bauer's \cite{bauer} computation in the ideal class group assuming the function field has a totally ramified infinite place is available.   The material presented here is the subject of the second author's Ph.D. thesis.

\bibliographystyle{amsplain}
\bibliography{mcom-l}

\end{document}